\documentclass[11pt,reqno]{amsart}

\usepackage{amsmath, amsfonts, amsthm, amssymb}

\numberwithin{equation}{section}
\newtheorem{Theorem}{Theorem}[section]
\newtheorem{Lemma}{Lemma}[section]
\newtheorem{Corollary}{Corollary}[section]
\theoremstyle{definition}
\newtheorem{Definition}{Definition}[section]
\theoremstyle{remark}
\newtheorem{Remark}{Remark}[section]
\newtheorem{Proposition}{Proposition}[section]

\renewcommand{\r}{\rho}

\def\i{\varepsilon}

\newcommand{\R}{{\mathbb R}}
\newcommand{\Dv}{{\rm div}}

\newcommand{\dl}{\delta}

\def\f{\frac}

\def\ov{\overline}

\def\D{\Delta }

\def\hf1{^\f{1}{1-\xi^2}}

\def\be{\begin{equation}}
\def\en{\end{equation}}
\def\bs{\begin{split}}
\def\es{\end{split}}

\author{Xianpeng Hu and Dehua Wang}
\address{Courant Institute of Mathematical Sciences, New York University, New York, NY 10012.}
\email{xianpeng@cims.nyu.edu}
\address{Department of Mathematics, University of Pittsburgh,
                           Pittsburgh, PA 15260.}
\email{dwang@math.pitt.edu}

\title[Weak Stability for the Vlasov-Maxwell-Boltzmann Equations]
{Weak Stability and Large Time Behavior for the Cauchy Problem of the Vlasov-Maxwell-Boltzmann Equations}

\date{\today}

\begin{document}

\begin{abstract}

The Cauchy problem for the Vlasov-Maxwell-Boltzmann equations (VMB) is
considered.  First the renormalized solution to the Vlasov equation with the Lorentz force is discussed and the difficulty
on the partial differentiability of the coefficients is overcome. Then the weak stability of the renormalized solutions to the Cauchy
problem of VMB is established using  the compactness of velocity averages and a renormalized formulation. 
Furthermore,  the large time behavior of the renormalized solutions
to VMB is studied and it is proved that the density of particles tends to a local Maxwellian as the time goes to infinity.
\end{abstract}
\maketitle

\section{Introduction}
Since the work of DiPerna and Lions \cite{DL4} on the Cauchy
problem for the Boltzmann equation twenty years ago, it has been a
well-known open problem to extend their theory to the
Vlasov-Maxwell-Boltzmann equations. Among  the difficulties, how
to define the characteristics of the Vlasov-Maxwell-Boltzmann
equations is a major obstacle. In this paper, we will give the following partial
results: the weak stability and large time behavior of the renormalized solutions to the
Vlasov-Maxwell-Boltzmann equations, and existence of the renormalized solutions
to the Vlasov equation with the Lorentz force. The fundamental model for dynamics of
dilute charged particles is described by the
Vlasov-Maxwell-Boltzmann equations (VMB) of the following form
\cite{CC, DD, RTG, YG,DL3, SR}:
\begin{subequations}\label{21}
\begin{align}
&\f{\partial f}{\partial t}+\xi\cdot\nabla_x f+(E+\xi\times
B)\cdot\nabla_\xi f=Q(f,f),\quad x\in \R^3,\quad \xi\in \R^3,\quad
t\ge 0,\label{e211}\\
&\f{\partial E}{\partial t}-\nabla\times B=-j, \quad \Dv B=0,\quad
\textrm{on}\quad \R^3_x\times(0,\infty),\label{e212}\\
&\f{\partial B}{\partial t}+\nabla\times E=0, \quad\Dv E=\r,\quad
\textrm{on}\quad \R^3_x\times(0,\infty),\label{e213}\\
&\r=\int_{\R^3}fd\xi, \quad j=\int_{\R^3}f\xi d\xi, 
\quad \textrm{on}\quad \R^3_x\times(0,\infty), \label{e214}
\end{align}
\end{subequations}
where  $f=f(t,x,\xi)$ is a nonnegative function for the density of particles which at
time $t$ and position $x$ move with velocity $\xi$ under the
Lorentz force $$E+\xi\times B,$$  $E$ is the electric field,  $B$ is the magnetic field,
the function $j$ is called the current density, and  the function $\r$ is the charge density.
The collison operator $Q(f,f)$,
which acts only on the velocity dependence of $f$ (this reflects
the physical assumption that collisions are localized in space and
time), is defined as
$$Q(f,f)=\int_{\R^3}d\xi_*\int_{S^2}d\omega\,
b(\xi-\xi_*,\omega)(f'f'_*-ff_*),$$ with $\omega\in S^2$, the unit
sphere in $\R^3$, where $b=b(z,\omega)$ denotes the collision kernel which is a given nonnegative function defined on $\R^3\times S^2$,  and
$$f_*=f(t,x,\xi_*),\quad f'=f(t, x,\xi'),\quad f'_*=f(t, x,\xi'_*),$$  with
\begin{equation*}
\begin{split}
&\xi'=\xi-(\xi-\xi_*,\omega)\omega,\\
&\xi_*'=\xi_*+(\xi-\xi_*,\omega)\omega,
\end{split}
\end{equation*}
which yield one convenient parametrization of the set of solutions to the law
of elastic collisions
\begin{equation*}
\begin{split}
&\xi'+\xi'_*=\xi+\xi_*,\\
&|\xi'|^2+|\xi'_*|^2=|\xi|^2+|\xi_*|^2.
\end{split}
\end{equation*}
  The interpretation of $\xi$,
$\xi_*$, $\xi'$, $\xi_*'$ is the following: $\xi, \xi_*$ are the
velocities of two colliding molecules immediately before collision,
while $\xi', \xi'_*$ are the velocities immediately after the
collision. Those unknown functions
$f$, $E$, and $B$ are strongly coupled, and the constraint on the
divergence of $E$ will be ensured provided that the conservation
of charge holds; that is,
$$\f{\partial \r}{\partial t}+\Dv_x j=0,$$ since
\begin{equation*}
\begin{split}
0=\f{\partial}{\partial t}(\Dv_x E-\r)&=\Dv_x E_t-\r_t\\
&=\Dv_x(\nabla_x\times B-j)-\r_t\\
&=-\r_t-\Dv_x j,
\end{split}
\end{equation*}
due to the fact $\Dv(\nabla\times v)=0$ for any vector-valued
function $v$. Similarly, the magnetic field $B$ remains divergence free if
it is so initially.

The VMB equations are integro-differential equations which
provide  a mathematical model for the statistical evolution of
dilute charged particles. The construction of  global solutions
to  VMB has been open for a long time until only a
few years ago. In Guo \cite{YG}, a unique global in time classical
solution near a global Maxwellian (independent of   space and
  time) was constructed. See also Strain \cite{SR} for the extension to
the Cauchy problem. Notice that, Lions constructed in \cite{DL3} a
very weak solution to VMB, which is usually called a
measure-valued solution, using Young's measure to deal with the
nonlinearity.

For the particles without collision (cf. \cite{FFM, DL1, LP, RTG, GR,
SJ}), or when the molecules are so rare that they do not interact with each
other, VMB becomes  the so-called Vlasov-Maxwell system (VM),
\begin{subequations}
\begin{align}
&\f{\partial f}{\partial t}+\xi\cdot\nabla_x f+(E+\xi\times
B)\cdot\nabla_\xi f=0,\quad x\in \R^3,\quad \xi\in \R^3,\quad
t\ge 0, \label{q1}\\
&\f{\partial E}{\partial t}-\nabla\times B=-j, \quad \Dv B=0,\quad
\textrm{on}\quad \R^3_x\times(0,\infty),\\
&\f{\partial B}{\partial t}+\nabla\times E=0, \quad\Dv E=\r,\quad
\textrm{on}\quad \R^3_x\times(0,\infty),\\
&\r=\int_{\R^3}fd\xi, \quad j=\int_{\R^3}f\xi
d\xi, \quad \textrm{on}\quad \R^3_x\times(0,\infty).
\end{align}
\end{subequations}
Note that \eqref{q1} is a transport equation with a divergence
free coefficient, that is
$$\Dv_{x,\xi}(\xi, E+\xi\times B)=0.$$
This property ensures that the solution will remain the same
integrability as the initial data. With the help of this
observation and velocity averaging lemma, DiPerna and Lions proved
in \cite{DL1} the global existence in time of weak solutions to VM with large
initial data. For the smooth solutions to VM, we refer the
 readers  to Glassey \cite{RTG} and Schaeffer \cite{SJ}.

The main goal of this paper is  to show the weak stability and the
large time behavior of the renormalized solutions to VMB. To this
end, we will need an existence result of the renormalized solution to
the Vlasov equation \eqref{q1}. Notice that the Vlasov equation is a
transport equation with only partially $W^{1,1}_{loc}$
regularity, since usually we can not expect  any differentiability
on the magnetic field $B$ and the electric field $E$ from the
conservation of energy. Inspirited by the result in Bouchut
\cite{FB} and Le Bris-Lions \cite{LL}, we will first show the
existence of renormalized solutions to the Vlasov equation. The
presence of a non-trivial magnetic field $B(x,t)$, a natural
consequence of the celebrated Maxwell theory for electromagnetism,
creates severe mathematical difficulty in studying the weak stability of
weak solutions and the construction of global in time solutions
for VMB. Our first result on weak stability is built on our above
mentioned new result about renormalized solutions to the Vlasov
equation with the aid of the velocity average lemma (DiPerna-Lions
\cite{DL1} and DiPerna-Lions-Meyer \cite{DLM}) and some techniques
from Lions \cite{DL2, DL3}.  Our  second result on renormalized
solutions to VMB is their large time behavior, since from the
physical point of view, the density of particles is assumed to
converge to an equilibrium represented by a Maxwellian function of
the velocity as the time $t$ becomes large. Our results heavily
depend on, apart from the weak compactness property,
\begin{itemize} 
\item the existence of renormalized solutions to the Vlasov equation;
\item a renormalized formulation, which is crucial to make sure that the
quadratic term $Q(f,f)$ is meaningful in $\mathcal{D}'$ (sense of distributions); and
\item the  velocity averaging lemma \cite{DL1, DLM}, which is
crucial for the convergence of nonnlinear term $(E+\xi\times B)\cdot\nabla_\xi f$.
\end{itemize}

The stability of renormalized solutions under weak convergence
yields a consequence on the propagation of smoothness for those
solutions. Indeed, a sequence of renormalized solutions
$\{f_n\}_{n=1}^\infty$ to VMB is relatively strongly compact in
$L^1([0,T]\times\R^6)$ if and only if the sequence of the
corresponding initial data $\{f_{0n}\}_{n=1}^\infty$ is relatively
strongly compact in $L^1(\R^6)$. In other words, under our
assumption on the collision kernel and the integrability of the
electric field and the magnetic field, {\textit{no oscillations
develop unless they are present from the beginning}}.

In order to prove our results,  the standard \textit{a priori} estimates derived from the conservation laws and H theorem are very useful,  and in addition we need some assumptions on the integrability of the electric field
$E(t,x)$ and the magnetic field $B(t,x)$. More precisely,  besides the standard estimate of $E$ and $B$ in
$L^\infty(0,T; L^2(\R^3))$, we need to assume that $E$ is
uniformly bounded in $L^\infty(0,T; L^5(\R^3))$ and $B$ is
uniformly bounded in $L^\infty(0,T; L^s(\R^3))$ for some $s>5$.
The reasons for these requirements on $E$ and $B$ are twofold: (I)
when we define the characteristics for the Vlasov equation, we need
a bound on $E$ in $L^\infty(0,T; L^5(\R^5))$; (II)  the averaging lemma (cf. \cite{DL3}), combining with the
uniform bound of $\int_{\R^3}fd\xi$ in $L^\infty(0,T;
L^{\f{5}{3}}(\R^3))$ and the uniform bound of $\int_{\R^3}\xi
fd\xi$ in $L^\infty(0,T; L^{\f{5}{4}}(\R^3))$, implies the
compactness of the first two moments of $f$ on $L^p(0,T;
L^p_{loc}(R^3))$ for any $1\le p<\f{5}{4}$, which is enough to
ensure the convergence of the nonlinear Lorentz force term in the
sense of the distributions provided that $E$ and $B$ are uniformly
bounded in $L^\infty(0,T; L^5(\R^3))$ and $L^\infty(0,T;
L^s(\R^3))$ for some $s>5$.

We now remark that throughout this work we never
claim the existence of renormalized solutions to VMB. Actually,
all results in this paper are based on the assumption of such an
exact existence or the existence of a sequence of approximating
solutions. One possible direction to address the existence problem
 may be based on the construction of a sequence of exact
solutions or approximating solutions with the requirement that
the electric field $E$ is uniformly bounded in $L^\infty(0,T; L^5(\R^3))$.
 We notice that the hyperbolic property of the Maxwell
equations also demonstrates some difficulties if we want to
improve the integrability of the electric field and the magnetic
field. How to fulfill this strategy is still an open question and
will be the topic of our future research.

When the Lorentz force disappears, that is $E+\xi\times B=0$,
VMB becomes the classical Boltzmann equation. For the Cauchy
problem of the classical Boltzmann equation, in \cite{DL4}
DiPerna and Lions proved the global existence of renormalized
solutions with augular cut-off collision kernel and
arbitrary initial data, see also \cite{AV, CC, DL, DL1, DL5, DL2,
DL3} and the references cited therein. Later, Hamdache extended
this existence result to a bounded  domain in \cite{HK}. The method
explored for the existence result was the analysis of the weak
stability of solutions. The argument strongly relied on some
compactness properties (see \cite{DL2}) which  hold
for sequences of renormalized solutions. 
In \cite{DL3}, Lions extended the similar weak stability and global
existence result to the Vlasov-Poission-Boltzmann equations. For the
extension to the Landau equation, see Villani \cite{VC}. For the long time
behavior of the Boltzmann equations, see \cite{DD, DL5, LD, VC}.

This paper will proceed as follows. We will discuss the
renormalized solution to the Vlasov equation in Section 2. Section 3 is
devoted to stating $\textit{a priori}$ estimates for VMB, main
assumptions and main results on the weak stability of renormalized
solutions to VMB.  Then,  Theorem \ref{2T1} on weak stability and Theorem
\ref{2T2} on the propagation of smoothness will be proved  in Section 4 and Section 5,
respectively. In Section 6, we study the large time behavior and establish mathematically the
convergence of $f$ to a local Maxwellian satisfying the
Vlasov-Maxwell equations. Finally, in Section 7, we explain an
extension of our results to the relativistic Vlasov-Maxwell-Boltzmann
equations.

\bigskip

\section{Renormalized Solutions to the Vlasov Equation}

In this section, we consider the Vlasov equation of the form:
\begin{equation}\label{11}
\partial_t f+\xi\cdot\nabla_x f+(E+\xi\times B)\cdot\nabla_\xi
f=0,
\end{equation}
with $B(x,t)\in L^\infty(0,T;L^2(\R^3_x))$ and $E(x,t)\in
L^\infty(0,T; L^2\cap L^5(\R^3_x))$.

If we set
$$y=(x, \xi)\in \R^6, \quad\quad \mathcal{B}=(\xi, E+\xi\times B)\in\R^6,$$
 then \eqref{11} becomes a standard transport equation
\begin{equation}\label{12}
\partial_t f+\mathcal{B}\cdot\nabla f=0.
\end{equation}
The question of whether the Vlasov equation has renormalized solutions  is
not only useful when the normalized solution to
VMB system is considered, but also has its
own interest due to the lower regularity of the coefficients. The
renormalized solutions mean that \eqref{12} still holds if we
replace $f$ by $\beta(f)$ with a suitable $\beta$. Over past
twenty years, there are many important progress about the
renormalized solutions to \eqref{12}. More precisely,
DiPerna and Lions showed in \cite{DL} the existence of renormalized solutions
when the coefficient $\mathcal{B}\in W^{1,1}(\R^6)$. In 2004,
Ambrosio extended the DiPerna-Lions theory to BV (bounded
variations) field in \cite{AL} (for related work, see \cite{FB}).
Also, in 2004 Le Bris and Lions extended in \cite{LL} the
DiPerna-Lions theory to the case that the coefficient has only
partial regularity.

For the VMB or the Vlasov equation, the velocity $\mathcal{B}$ is
no longer in $W^{1,1}_{(x,\xi),loc}$. Inspired by \cite{FB, LL},
we claim that we still can prove the existence of a renormalized
solution to \eqref{11} under the conditions that $E(x,t)\in
L^\infty(0,T; L^2(\R^3)\cap L^5(\R^3))$, and $B(x,t)\in
L^\infty(0,T; L^2(\R^3))$. This is a crucial step for establishing
renormalized solutions to the  Vlasov-Maxwell-Boltzmann equations.

\begin{Theorem}\label{1T1}
Assume that $B(x,t)\in L^\infty(0,T; L^2(\R^3_x))$ and $E(x,t)\in
L^\infty(0,T; L^2\cap L^5(\R^3_x))$. Let $f_0\in L^1\cap
L^\infty(\R^6)$ and $|\xi|^2f_0\in L^1(\R^6)$. Then there exists a
solution to \eqref{11} (and hence to \eqref{12}) such that
\begin{equation*}
\begin{split}
f(t,x,\xi)\in L^\infty([0,T], L^1_{x,\xi}\cap
L^\infty_{x,\xi}(\R^3)),
\end{split}
\end{equation*}
and $|\xi|^2f\in L^\infty(0,T; L^1(\R^6))$,
satisfying the initial condition $f|_{t=0}=f_0(x,\xi)$.
Furthermore, if $f_0\in L^\infty_x(L^1_\xi(\R^3))$, then $f\in
L^\infty(0,T; L^\infty_x(L^1_\xi(\R^3)))$, and hence the solution
is unique.
\end{Theorem}

To begin with the proof, notice that
$\mathcal{B}=(\mathcal{B}_1,\mathcal{B}_2)$ satisfies
$$\Dv_x{\mathcal{B}_1}=\Dv_{\xi}{\mathcal{B}_2}=0,$$
with
$$\mathcal{B}_1(x,\xi)=\xi\in W^{1,1}_{\xi,loc}(\R^3) \quad(\textrm{it does not depend on} \; x),$$
$$\mathcal{B}_2(x,\xi)=E+\xi\times B\in L^1_{x, loc}(\R^3, W^{1,1}_{\xi, loc}(\R^3)).$$
The proof of this theorem is divided into three steps. The
uniqueness  is a crucial issue which is the consequence of
the following two lemmas, the first one dealing with
regularization, and the second one stating the uniqueness. Finally, we will
show the existence part.

Now we denote the mollifier $\kappa_\i$ as
$$\kappa_{\varepsilon}=\f{1}{\varepsilon^n}\kappa\left(\f{\cdot}{\varepsilon}\right),
\qquad\kappa\in \mathcal{D}(\R^3),\qquad
\int_{\R^3}\kappa=1,\qquad \kappa\ge 0,$$
where $\mathcal{D}(\R^3)=C_0^\infty(\R^3)$.
Then, we have the following two lemmas.

\begin{Lemma}\label{1L1}
Let $f=f(t,x,\xi)\in L^\infty([0,T], L^1_{(x,\xi)}\cap
L^\infty_{(x,\xi)}(\R^6))$ be a solution of \eqref{12}, and
$\kappa_{\varepsilon}$ and $\kappa_{\mu}$ be two regularizations
with two different scalings, respectively, in the variable $x$ and
$\xi$. Then,  for any $\varepsilon>0$, there exists a number $\mu(\varepsilon)$ with
$0<\mu(\varepsilon)\le\varepsilon^2$ such that
$$f_{\varepsilon,\mu(\varepsilon)}=(f*\kappa_{\varepsilon})*\kappa_{\mu(\varepsilon)}$$
is a smooth (in $(x,\xi)$) solution of
$$\f{\partial f_{\varepsilon, \mu(\varepsilon)}}{\partial
t}+\mathcal{B}\cdot\nabla f_{\varepsilon,
\mu(\varepsilon)}=\mathcal{A}_{\varepsilon},$$ with
$$\lim_{\varepsilon\rightarrow
0}\mathcal{A}_{\varepsilon}=0,\quad\textrm{in}\quad
L^\infty\big([0,T], L^1_{(x,\xi),loc}\cap L^\infty_{(x,\xi),
loc}(\R^6)\big).$$
\end{Lemma}

\begin{Lemma}\label{1L2}
Let $f=f(t,x,\xi)\in L^\infty\big([0,T], L^1_{(x,\xi)}\cap
L^\infty_{(x,\xi)}(\R^3)\big)$ be a nonnegative solution of
\eqref{12} with zero initial data $f_0=0$.  If,  in addition,
$|\xi|^2 f\in L^\infty([0,T],L^1(\R^6))$ and $f\in
L^2_x(L^1_{\xi})$, then $f=0$ for all time $t>0$.
\end{Lemma}

We now prove   these two lemmas, and then finally complete
the proof of Theorem \ref{1T1}.

\subsection{Proof of Lemma \ref{1L1}}

We will use the mollifier to regularize the function $f$ in $\xi$ and $x$, while we assume  that $f$ is differentiable with respect to $t$ (the results below hold also for the general case from a standard mollification in $t$ with the help of Lebesgue's dominated theorem.)
All the functional spaces used here
are $\textit{local}$, which is clearly enough for such a
regularization result.

We first regularize in the $\xi$ variable by convoluting
\eqref{12} with $\kappa_{\mu}$ to get
\begin{equation}\label{13}
\f{\partial f*\kappa_{\mu}}{\partial t}+(\xi\cdot\nabla_x
f)*\kappa_{\mu}+((E+\xi\times B)\cdot\nabla_\xi f)*\kappa_{\mu}=0.
\end{equation}
Denoting by
$$[(E+\xi\times B)\cdot\nabla_\xi,
\kappa_{\mu}](f)=(E+\xi\times
B)\cdot\nabla_\xi(f*\kappa_{\mu})-\kappa_{\mu}*((E+\xi\times
B)\cdot\nabla_\xi f).$$ Then, \eqref{13} can be rewritten as
\begin{equation}\label{14}
\begin{split}
\f{\partial f*\kappa_{\mu}}{\partial t}+(\xi\cdot\nabla_x
f)*\kappa_{\mu}+(E+\xi\times
B)\cdot\nabla_\xi(f*\kappa_{\mu})=[(E+\xi\times B)\cdot\nabla_\xi,
\kappa_{\mu}](f).
\end{split}
\end{equation}

It is a standard fact (see \cite{DL}) that
\begin{equation}\label{16}
I_1^\mu:=[(E+\xi\times B)\cdot\nabla_\xi,
\kappa_{\mu}](f)\rightarrow 0 \quad\textrm{in}\quad
L^1_{(x,\xi),loc}
\end{equation}
as $\mu\rightarrow 0$. Indeed, this is clear for smooth coefficients
and $f$, while the general case follows as in \cite{DL} by dense property
through the estimate
$$\|[(E+\xi\times B)\cdot\nabla_\xi,
\kappa_{\mu}](f)\|_{L^1_{\xi,loc}}\le C\|E+\xi\times
B\|_{W^{1,1}_{\xi,loc}}\|f\|_{L^\infty_{\xi}},$$ which then
implies the following standard estimate by integrating in $x$, 
$$\|[(E+\xi\times B)\cdot\nabla_\xi,
\kappa_{\mu}](f)\|_{L^1_{(x,\xi),loc}}\le C\|E+\xi\times
B\|_{L^1_{x,loc}(W^{1,1}_{\xi,loc})}\|f\|_{L^\infty_{x,\xi}}.$$

Next, we regularize in the $x$ variable by convoluting \eqref{14}
with $\kappa_\varepsilon$ for $f_\mu=f*\kappa_\mu$ to obtain,
\begin{equation}\label{15}
\begin{split}
&\f{\partial f_\mu*\kappa_{\varepsilon}}{\partial
t}+(\xi\cdot\nabla_x
f)*\kappa_{\mu}*\kappa_\varepsilon+(E+\xi\times
B)\cdot\nabla_\xi(f_\mu*\kappa_{\varepsilon})\\&\quad=[(E+\xi\times
B)\cdot\nabla_\xi,
\kappa_{\varepsilon}](f_\mu)+I_1^\mu*\kappa_\varepsilon.
\end{split}
\end{equation}

We now successively deal with each terms on the right-hand side of
\eqref{15}. First, it is easy to observe that for fixed $\mu$, we
have
$$I_1^\mu*\kappa_\varepsilon\rightarrow I_1^\mu,\quad\textrm{as}\quad
\varepsilon\rightarrow 0,$$ in $L^1_{(x,\xi), loc}$, which together with \eqref{16} implies that
\begin{equation}\label{17}
\lim_{\mu\rightarrow 0}\lim_{\varepsilon\rightarrow 0}
I_1^\mu*\kappa_\varepsilon=0,\quad\textrm{in}\quad
L^1_{(x,\xi),loc}.
\end{equation}
Second, for the first term on the right-hand side of \eqref{15},
we have
\begin{equation*}
\begin{split}
[(E+\xi\times B)\cdot\nabla_\xi, \kappa_{\varepsilon}](f_\mu)
&=(E+\xi\times
B)\cdot\nabla_\xi(\kappa_\varepsilon*f_\mu)-\kappa_\varepsilon*\big((E+\xi\times
B)\cdot\nabla_\xi f_\mu\big)\\
&=(E+\xi\times B)\cdot\big((\nabla_\xi
f_\mu)*\kappa_\varepsilon\big)-\kappa_\varepsilon*\big((E+\xi\times
B)\cdot\nabla_\xi f_\mu\big)\\
&=\big[(E+\xi\times B),\kappa_\varepsilon\big](\nabla_\xi f_\mu).
\end{split}
\end{equation*}
The latter bracket can be controlled as follows:
$$\Big\|\big[(E+\xi\times B),\kappa_\varepsilon\big](\nabla_\xi
f_\mu)\Big\|_{L^1_{(x,\xi),loc}}\le C\|E+\xi\times
B\|_{L^1_{(x,\xi), loc}}\|\nabla_\xi
f_\mu\|_{L^\infty_{(x,\xi)}}.$$ Hence, for fixed $\mu$, we have
\begin{equation*}
\lim_{\varepsilon\rightarrow 0}[(E+\xi\times B)\cdot\nabla_\xi,
\kappa_{\mu}](f_\mu)=0,
\end{equation*}
in $L^1_{(x,\xi),loc}.$ This implies,
\begin{equation}\label{18}
\lim_{\mu\rightarrow 0}\lim_{\varepsilon\rightarrow
0}[(E+\xi\times B)\cdot\nabla_\xi, \kappa_{\mu}](f_\mu)=0,
\end{equation}
in $L^1_{(x,\xi),loc}.$ By a standard diagonization procedure, for
any $\varepsilon>0$, we can find $\mu(\varepsilon)$ with $0<\mu(\varepsilon)\le
\varepsilon^2\rightarrow 0$ such that
\begin{equation*}
\lim_{\varepsilon\rightarrow 0}
I_1^{\mu(\varepsilon)}*\kappa_\varepsilon=0,\quad\textrm{in}\quad
L^1_{(x,\xi),loc}.
\end{equation*}
and
\begin{equation*}
\lim_{\varepsilon\rightarrow 0}[(E+\xi\times B)\cdot\nabla_\xi,
\kappa_{\mu(\varepsilon)}](f_{\mu(\varepsilon)})=0,\quad\textrm{in}\quad
L^1_{(x,\xi),loc}.
\end{equation*}

To complete the proof of this lemma,  it remains to show the following convergence for the above chosen
$\mu(\varepsilon)$:
$$I_2^{\mu(\varepsilon),\varepsilon}=(\xi\cdot\nabla_x
f)*\kappa_{\mu(\varepsilon)}*\kappa_\varepsilon-\xi\cdot\nabla_x(f_{\mu(\varepsilon)}*\kappa_\varepsilon)$$
in $L^1_{(x,\xi),loc}$. Indeed, we can control
$I^{\mu(\varepsilon),\varepsilon}_2$ as
\begin{equation*}
\begin{split}
\left|I^{\mu(\varepsilon),\varepsilon}_2\right|&=\left|\int_{\R^6}\big[(\xi-\eta)\cdot\nabla_x
f(x-\zeta, \xi-\eta)-\xi\cdot\nabla_x f(x-\zeta,
\xi-\eta)\big]\kappa_{\mu(\varepsilon)}\kappa_\varepsilon d\zeta d\eta\right|\\
&=\left|\int_{\R^6}\big[\eta\cdot\nabla_x
f(x-\zeta,\xi-\eta)\big]\kappa_{\mu(\varepsilon)}\kappa_\varepsilon d\zeta d\eta\right|\\
&=\left|\int_{\R^6}\eta\cdot\nabla_\zeta\kappa_\varepsilon(\zeta)\kappa_{\mu(\varepsilon)}(\eta)
f(x-\zeta, \xi-\eta)d\zeta d\eta\right|\\
&\le
\f{\mu(\varepsilon)}{\varepsilon}\int_{\R^6}|\varepsilon\nabla_\zeta\kappa_\varepsilon|
\kappa_{\mu(\varepsilon)}|f(x-\zeta, \xi-\eta)|d\zeta d\eta.
\end{split}
\end{equation*}
Thus, we deduce that, for any compact subset $K\subset \R^6$, by
Fubini's theorem,
\begin{equation}\label{19}
\begin{split}
\|I^{\mu(\varepsilon),\varepsilon}_2\|_{L^1(K)}&=\int_K\left|\int_{\R^6}\eta\cdot\nabla_\zeta\kappa_\varepsilon(\zeta)\kappa_{\mu(\varepsilon)}(\eta)
f(x-\zeta, \xi-\eta)d\zeta d\eta\right|dx d\xi\\
&\le C
\f{\mu(\varepsilon)}{\varepsilon}\Big(\int_{\R^3}|\varepsilon\nabla_\zeta\kappa_\varepsilon|
d\zeta\Big)\sup_{|\zeta|\le\varepsilon,
|\eta|\le\mu(\varepsilon)}\|f(x-\zeta,
\xi-\eta)\|_{L^1(K)}\\
&\le
C\f{\mu(\varepsilon)}{\varepsilon}\sup_{|\zeta|\le\varepsilon,
|\eta|\le\mu(\varepsilon)}\|f(x-\zeta, \xi-\eta)\|_{L^1(K)}.
\end{split}
\end{equation}

Since $f\in L^1_{(x,\xi)}$, one has, according to the
continuity of translation in $L^1(K)$,
$$\sup_{|\zeta|\le\varepsilon,
|\eta|\le\mu(\varepsilon)}\|f(x-\zeta,
\xi-\eta)-f(x,\xi)\|_{L^1(K)}\rightarrow 0,\quad
\textrm{as}\quad\varepsilon\rightarrow 0,$$ and hence,
$$\sup_{|\zeta|\le\varepsilon,
|\eta|\le\mu(\varepsilon)}\|f(x-\zeta,
\xi-\eta)\|_{L^1(K)}\quad\textrm{is uniforly bounded for all}\quad
\i\le 1.$$ Thus, if we let $\varepsilon \rightarrow 0$ and
$0\le\mu(\i)\le \i^2$, we deduce from \eqref{19} that
\begin{equation}\label{110}
I^{\mu(\varepsilon),\varepsilon}_2\rightarrow 0,\quad
\textrm{in}\quad
L^1_{(x,\xi),loc}\quad\textrm{as}\quad\i\rightarrow 0.
\end{equation}

Therefore, the lemma follows from  \eqref{17}, \eqref{18},  and \eqref{110}, and we complete the
proof of this lemma.
\qed
\bigskip

Next, we turn to the proof of Lemma \ref{1L2}.

\subsection{Proof of Lemma \ref{1L2}}

Let $f$ be a nonnegative solution as claimed in Theorem \ref{1T1}.
We introduce two cut-off functions, respectively, with respect to
each variable $x$ and $\xi$. For $m, n\in \bf{N}$, we denote them
by $$\psi_m(x)=\psi\left(\f{x}{m}\right), \quad\text{and}\quad \phi_n(\xi)=\phi\left(\f{\xi}{n}\right),$$
where $\psi\in\mathcal{D}(\R^3)$, $\psi\ge 0$, $\psi=1$ for
$|x|\le 1$ and $\psi=0$ for $|x|\ge 2$; and the analogous properties
are required on $\phi$ with respect to the variable $\xi$. We
first multiply \eqref{11} by $\phi_n$ and integrate over $\xi$
space to obtain
\begin{equation}\label{111}
\f{\partial}{\partial t}\int_{\R^3} f\phi_n
d\xi+\Dv_x\left(\int_{\R^3}\xi f\phi_n
d\xi\right)+\int_{\R^3}(E+\xi\times B)\cdot\nabla_\xi f\phi_n
d\xi=0.
\end{equation}

For the last term in \eqref{111}, we deduce, due to
$\Dv_\xi(E+\xi\times B)=0$,
\begin{equation*}
\begin{split}
\int_{\R^3}(E+\xi\times B)\cdot\nabla_\xi f\phi_n
d\xi&=-\int_{\R^3}f (E+\xi\times B)\cdot\nabla_\xi \phi_n d\xi\\
&=-\int_{\R^3}f\f{1+|\xi|}{n}\f{E+\xi\times
B}{1+|\xi|}\cdot\nabla_\xi\phi\left(\f{\xi}{n}\right)d\xi.
\end{split}
\end{equation*}
Now we multiply \eqref{111} by $\psi_m$ and integrate over $x$
space to deduce
\begin{equation}\label{112}
\begin{split}
&\f{d}{dt}\int_{\R^{6}}f\psi_m\phi_n
dxd\xi+\int_{\R^3}\psi_m\Dv_x\left(\int_{\R^3}\xi f\phi_n
d\xi\right)
dx\\&\quad=\int_{\R^3}\psi_m\int_{\R^3}f\f{1+|\xi|}{n}\f{E+\xi\times
B}{1+|\xi|}\cdot\nabla_\xi\phi\left(\f{\xi}{n}\right)d\xi dx.
\end{split}
\end{equation}
Hence, using the integration by parts for the second term in
\eqref{112}, we have
\begin{equation}\label{113}
\begin{split}
&\f{d}{dt}\int_{R^{2n}}f\psi_m\phi_n
dxd\xi-\int_{\R^3}\nabla_x\psi_m\cdot\left(\int_{\R^3}\xi f\phi_n
d\xi\right)
dx\\&\quad=\int_{\R^3}\psi_m\int_{\R^3}f\f{1+|\xi|}{n}\f{E+\xi\times
B}{1+|\xi|}\cdot\nabla_\xi\phi\left(\f{\xi}{n}\right)d\xi dx.
\end{split}
\end{equation}

Next, we proceed to control the  two integral terms in \eqref{113}. Indeed, for the second term in \eqref{113}, we have
\begin{equation}\label{114}
\begin{split}
\left|\int_{\R^3}\nabla_x\psi_m\cdot\left(\int_{\R^3}\xi f\phi_n
d\xi\right)
dx\right|&=\left|\int_{\R^3}\f{1}{m}\nabla_x\psi\left(\f{x}{m}\right)\cdot\left(\int_{\R^3}\xi
f\phi_n d\xi\right) dx\right|\\
&\le C\f{1}{m}\|\xi f\chi_{\{m\le|x|\le 2m, |\xi|\le 2n\}}
\|_{L^1_{(x,\xi)}}\rightarrow 0,
\end{split}
\end{equation}
as $m\rightarrow \infty$ and $n\rightarrow\infty$. Here we used
$\xi f\in L^\infty([0,T], L^1(\R^6))$, because
\begin{equation*}
\begin{split}
\int_{\R^3}\int_{\R^3}|\xi| f d\xi dx&\le
R\int_{\R^3}\int_{\R^3\cap \{|\xi|\le R\}}fd\xi
dx+\int_{\R^3}\int_{\R^3\cap\{|\xi|>R\}}|\xi|f d\xi dx\\
&\le R\int_{\R^3}\int_{\R^3\cap \{|\xi|\le R\}}fd\xi
dx+\f{1}{R}\int_{\R^3}\int_{\R^3\cap\{|\xi|>R\}}|\xi|^2f d\xi dx\\
&\le R\|f\|_{L^1(\R^6)}+\f{1}{R}\||\xi|^2 f\|_{L^1(\R^6)}\\
&\le 2\|f\|_{L^1(\R^6)}^{\f{1}{2}}\||\xi|^2
f\|_{L^1(\R^6)}^{\f{1}{2}},
\end{split}
\end{equation*}
by optimizing the value of $R$.

On the other hand, for $m$  fixed, we claim that the
term on the right-hand side of \eqref{113} goes to zero as $n$
goes to infinity by Lebesgue's dominated convergence theorem.
Indeed, as $\nabla\phi$ is $L^\infty$ and supported in the annular
$\{1\le |\xi|\le 2\}$, we have for almost all $x\in \R^3$,
\begin{equation*}
\begin{split}
&\left|\psi_m\int_{\R^3}f\f{1+|\xi|}{n}\f{E+\xi\times
B}{1+|\xi|}\cdot\nabla_\xi\phi\left(\f{\xi}{n}\right)d\xi\right|\\
&\quad\le \psi_m\int_{\R^3}f\f{1+|\xi|}{n}\f{|E+\xi\times
B|}{1+|\xi|}\left|\nabla_\xi\phi\left(\f{\xi}{n}\right)\right|d\xi\\
&\quad\le 2\|\nabla\phi\|_{L^\infty}\psi_m\|f\chi_{\{n\le|\xi|\le
2n\}}\|_{L^1_{\xi}}(|E|+|B|)\\
&\quad\rightarrow 0,
\end{split}
\end{equation*}
as $n\to\infty$, since for almost all $x\in\R^3$,
$f(x,\cdot)\in L^1(\R^3_\xi)$. In addition, by the Cauchy-Schwarz
inequality, we have,
\begin{equation*}
\begin{split}
&\left|\psi_m\int_{\R^3}f\f{1+|\xi|}{n}\f{E+\xi\times
B}{1+|\xi|}\cdot\nabla_\xi\phi\left(\f{\xi}{n}\right)d\xi\right|\\
&\quad\le \psi_m\int_{\R^3}f\f{1+|\xi|}{n}\f{|E+\xi\times
B|}{1+|\xi|}\left|\nabla_\xi\phi\left(\f{\xi}{n}\right)\right|d\xi\\
&\quad\le 2\|\nabla\phi\|_{L^\infty}\|f\|_{L^1_{\xi}}(|E|+|B|)\\
&\quad\le
4\|\nabla\phi\|_{L^\infty}(\|f\|^2_{L^1_{\xi}}+|E|^2+|B|^2).
\end{split}
\end{equation*}
and the right-hand side is in $L^1_x$, since $f\in
L^2_x(L^1_{\xi})$ and $E, B\in L^2_{x}$. Thus, Lebesgue's theorem
applies and we get the convergence of the term on the right-hand
side of \eqref{113} to zero as $n$ goes to infinity, and $m$ being
kept fixed.

Collecting the behaviors of those two terms, we obtain with
\eqref{113}, as $n$, and next $m$, go to infinity,
$$\f{d}{dt}\int_{\R^{6}}f dxd\xi=0.$$
As $f_0=0$, this yields $f=0$ for all $t$ since $f\ge 0$ and this
concludes the proof.
\qed
\bigskip

Having proved Lemma \ref{1L1} and Lemma \ref{1L2}, we are now
ready to complete the proof of Theorem \ref{1T1} as follows.

\subsection{Proof of Theorem \ref{1T1}}

Assume for the
time being that we have at hand two solutions $f_1$ and $f_2$ to
\eqref{11} satisfying the regularity stated in Theorem \ref{1T1},
and sharing the same initial value. In view of the interpolation
between $L^1$ and $L^\infty$, and the fact $$f_i\in
L^\infty([0,T], L^1_{x,\xi}\cap L^\infty_{x,\xi}(\R^3))\cap
L^\infty\Big([0,T], L^\infty_x\big(\R^3,
L^1_{\xi}(\R^3)\big)\Big)$$ for $i=1,2$, we deduce that
$$f_i\in L^\infty\Big([0,T],
L^2_x\big(\R^3, L^1_{\xi}(\R^3)\big)\Big).$$  By virtue of Lemma
\ref{1L1}, their difference ${\tt f}=f_1-f_2$ satisfies
\begin{equation}\label{115}
\f{\partial {\tt f}_{\mu(\varepsilon), \varepsilon}}{\partial
t}+\mathcal{B}\cdot\nabla_{(x,\xi)}{\tt f}_{\mu(\varepsilon),\varepsilon}=\ov{\mathcal{A}}_\varepsilon,
\end{equation}
with the same notation as in Lemma \ref{1L1}. Since
${\tt f}_{\mu(\varepsilon),\varepsilon}\in C^\infty(\R^{6})$, we
multiply \eqref{115} by $\beta'({\tt f}_{\mu(\varepsilon),\varepsilon})$
for some function $\beta\in C^1(\R)$ with $\beta'$ bounded, and
obtain
$$\f{\partial \beta({\tt f}_{\mu(\varepsilon), \varepsilon})}{\partial
t}+\mathcal{B}\cdot\nabla_{(x,\xi)}
\beta({\tt f}_{\mu(\varepsilon),\varepsilon})=\ov{\mathcal{A}}_\varepsilon\beta'({\tt f}_{\mu(\varepsilon),\varepsilon}).$$
By letting $\varepsilon$ go to zero, we obtain the equation
$$\f{\partial\beta({\tt f})}{\partial t}+\mathcal{B}\cdot\nabla_{(x,\xi)}\beta({\tt f})=0,$$ in
$L^\infty([0,T]; L^1\cap L^\infty_{(x,\xi),loc})$ for such
functions $\beta$. Now, letting $\beta$ approximate the absolute
value function, we end up with
\begin{equation*}
\f{\partial |{\tt f}|}{\partial t}+\mathcal{B}\cdot\nabla_{(x,\xi)}|{\tt f}|=0.
\end{equation*}
This implies that we have a nonnegative solution $|{\tt f}|$ to
\eqref{11}, which vanishes at initial time and belongs to the
functional space stated in Lemma \ref{1L2}. Applying Lemma
\ref{1L2}, we get $|{\tt f}|=0$, that is, $f_1=f_2$. There remains now
to prove the existence part.

Existence in the functional space $L^\infty\big([0,T];
L^1_{(x,\xi)}\cap L^\infty_{(x,\xi)}(\R^{6})\big)$ is given in a
straight forward way by an application of Proposition 2.1 of
\cite{DL}. For the sake of consistency, let us only mention here
that it is a simple matter of regularization of the vector field
$\mathcal{B}$ appearing in \eqref{11}. That is, one introduces the solution
$f_\alpha$ to
$$\f{\partial f_\alpha}{\partial t}+\mathcal{B}_\alpha\cdot\nabla
f_\alpha=0,\quad \textrm{in}\quad (0,\infty)\times \R^{6},$$ where
$\mathcal{B}_\alpha=\kappa_\alpha*\mathcal{B}\in L^1([0,T];
C^\infty(\R^{6}))$ converges to $\mathcal{B}$, then shows the
desired  estimates on $f_\alpha$,  and finally passes to the limit.

Next, the non-standard part we have to prove here is the fact
that such a solution necessarily satisfies $|\xi|^2 f\in
L^\infty([0,T], L^1(\R^{6})$. This is actually a consequence of
the specific form of the transport equation and of the
regularization process we have already done. Indeed, first, by the method of
characteristics, we know if $f_0\ge 0$ $a.e$ in $\R^{6}$, then
$f(t)\ge 0$ $a.e$ in $\R^{6}$ for all $t\ge0$. Then, formally we
multiply \eqref{11} by $|\xi|^2$ to obtain
$$\f{\partial (|\xi|^2 f)}{\partial t}+|\xi|^2 \xi\cdot\nabla_x
f+|\xi|^2(E+\xi\times B)\cdot\nabla_\xi f=0.$$ Then we integrate
the above identity over $\xi$ on $\R^3$ to deduce
\begin{equation}\label{116}
\f{d}{dt}\int_{\R^6}|\xi|^2
fdxd\xi=\int_{\R^{6}}\Dv_\xi(|\xi|^2(E+\xi\times B)) fdxd\xi,
\end{equation}
since $$\int_{\R^{6}}|\xi|^2 \xi\cdot\nabla_x f
dxd\xi=-\int_{\R^{6}}\Dv_x(|\xi|^2 \xi) f dxd\xi=0.$$ For the term
on the right-hand side of \eqref{116}, we have
$$\int_{\R^{6}}\Dv_\xi(|\xi|^2(E+\xi\times B)) fdxd\xi=2\int_{\R^{6}}\left(\xi\cdot Ef\right)dxd\xi,$$
since
$$\Dv_\xi(|\xi|^2\xi\times B)=2\xi\cdot(\xi\times
B)+|\xi|^2\Dv_\xi(\xi\times B)=0.$$
Also, notice that, for $a.e$ $x\in \R^3$,
\begin{equation}\label{1116}
\begin{split}
\int_{\R^3}|\xi|fd\xi&\le\int_{\{|\xi|\le
R\}}Rfd\xi+\int_{\{|\xi|>R\}}|\xi|fd\xi\\
&\le\omega_3
R^4\|f\|_{L^\infty(\R^6)}+R^{-1}\int_{\{|\xi|>R\}}|\xi|^2 fd\xi\\
&\le C\left(\int_{\R^3}|\xi|^2fd\xi\right)^{\f{4}{5}},
\end{split}
\end{equation}
where $\omega_3$ is the volume of the unit ball in $\R^3$,  and in
the last inequality $R$ is taken to be
$$R=\left(\int_{\R^3}|\xi|^2fd\xi\right)^{\f{1}{5}}.$$
Hence, we have the following estimate, by the H\"{o}lder inequality,
\begin{equation*}
\begin{split}
\left|\int_{\R^{6}}\Dv_\xi(|\xi|^2(E+\xi\times B))
fdxd\xi\right|&=2\left|\int_{\R^{6}}\left(\xi\cdot
Ef\right)dxd\xi\right|\\
&\le 2\int_{\R^3}\int_{\R^3}|\xi|f|E|dxd\xi\\
&\le C\int_{\R^3}\left(\int_{\R^3}|\xi|^2fd\xi\right)^{\f{4}{5}}|E|dx\\
&\le
C\left(\int_{\R^{6}}|\xi|^2fdxd\xi\right)^{\f{4}{5}}\|E\|_{L^\infty([0,T],
L^5(\R^3))}.
\end{split}
\end{equation*}
Substituting this back to \eqref{116}, we obtain
$$\f{d}{dt}\int_{\R^6}|\xi|^2
fdxd\xi\le
C\left(\int_{\R^{6}}|\xi|^2fdxd\xi\right)^{\f{4}{5}}\|E\|_{L^\infty([0,T],
L^5(\R^3))}.$$ This implies
$$\int_{\R^6}|\xi|^2
f(t)dxd\xi\le \int_{\R^6}|\xi|^2 f_0dxd\xi+CT^5,$$ for all $t\in
[0,T].$

Finally, we show that the solution $f$ necessarily belongs to
$L^\infty([0,T]; L^\infty_x(\R^3, L^1_\xi(\R^3)))$ if $f_0\in
L^\infty_x(\R^3, L^1_\xi(\R^3))$. This is actually also a consequence
of the specific form of the transport equation and of the
regularization process as mentioned earlier. Indeed, we mollify $E+\xi\cdot B$ by
$\kappa_\alpha$ to obtain,
\begin{equation}\label{fe1}
\f{\partial f_{\alpha}}{\partial t}+\xi\cdot\nabla_x
f_{\alpha}+(E+\xi\times B)_\alpha\cdot\nabla_\xi f_{\alpha}=0.
\end{equation}
Integrating \eqref{fe1} over $\xi$ in $\R^3$, one has, thanks to
the fact that $\Dv_\xi(E+\xi\times B)_\alpha=0$,
\begin{equation}\label{fe2}
\f{\partial}{\partial
t}\int_{\R^3}f_{\alpha}d\xi+\xi\cdot\nabla_x\int_{\R^3}
f_{\alpha}d\xi=0.
\end{equation}
That is equivalent to saying that $\int_{\R^3} f_{\alpha}d\xi$
satisfies a conservation form, which yields the conservation
over time of
$$\left\|\int_{\R^3}f_\alpha d\xi\right\|_{L^\infty_x}.$$
Hence, $f_\alpha \in L^\infty(0,T; L^\infty_x(\R^3,
L^1(\R^3_\xi)))$. By letting $\alpha\rightarrow 0$, one obtain $$f
\in L^\infty(0,T; L^\infty_x(\R^3, L^1(\R^3_\xi))).$$

The proof of Theorem \ref{1T1} of complete.
 \qed

\begin{Remark}
The assumption $E\in L^\infty(0,T; L^5(\R^3))$ is only needed to
show the uniform estimate $|\xi|^2f\in L^\infty(0,T; L^1(\R^6))$.
\end{Remark}

We now turn to the extension of the previous result to
less regular initial data through the notion of renormalized
solutions in the spirit of \cite{DL}. As in \cite{DL}, we
consider the set $L^0$ of all measurable functions $f$ on $\R^6$
with value in $\ov{\R}$ such that $meas\{|f|>\lambda\}<\infty,$
for all $\lambda>0$. For any $\beta\in C_{0,0}(\R)$, bounded and
vanishing near zero, we thus have $\beta(f)\in L^1\cap
L^\infty(\R^6)$ for any $f\in L^0$. As in \cite{DL}, we shall say
that a sequence $f_n$ is bounded (respectively, converges) in
$L^0$ whenever $\beta(f_n)$ is bounded (respectively, converges)
in $L^1$ for any such $\beta$. But now we need some additional
assumptions on our initial data, and that is why we consider the
subset $L^{00}$ of $L^0$ consisting of functions $f$ satisfying
$$\int_{\{|f(x,\xi)|>\dl\}}|\xi|^2 dxd\xi\le c_{\dl}<\infty, \qquad \forall\,\dl>0 .$$ This subset is equipped with the topology
induced by that of $L^0$. For any $f\in L^{00}$, we have
$|\xi|^2\beta(f)\in L^1_{(x,\xi)}(\R^6)$. Indeed, for $\dl$ small
enough such that $\beta$ vanishes on $[0,\dl]$, we have
\begin{equation*}
\begin{split}
\int_{\R^6}|\xi|^2|\beta(f(x,\xi))|dxd\xi&=\int_{\{
|f(x,\xi)|>\dl\}}|\xi|^2|\beta(f(x,\xi))|dxd\xi\\
&\quad+\int_{\{|f(x,\xi)|\le\dl\}}|\xi|^2|\beta(f(x,\xi))|dxd\xi\\
&\le \|\beta\|_{L^\infty}c_{\dl}+0<\infty.
\end{split}
\end{equation*}

It follows that if we choose $f_0$ in $L^{00}$, then $\beta(f_0)$
is a convenient initial condition for the transport equation
considered in Theorem \ref{1T1}. We therefore say that $f$ is a
renormalized solution of \eqref{11} complemented by an initial
condition $f_0\in L^{00}$ whenever $\beta(f)$ is a solution of
\eqref{11} in the sense of Theorem \ref{1T1} with the initial
condition $\beta(f_0)$.

\bigskip

\section{Stability of Vlasov-Maxwell-Boltzmann Equations: Main Results}

Let us begin by recalling that the general Vlasov-Maxwell-Boltzmann equations \eqref{21}
has the collison operator $Q(f,f)$ which can be written as
$$Q(f,f)=Q^+(f,f)-Q^-(f,f),$$
where
$$Q^+(f,f)=\int_{\R^3}d\xi^*\int_{S^2}d\omega b(\xi-\xi_*,\omega)f'f'_*,$$
and
$$Q^-(f,f)=\int_{\R^3}d\xi_*\int_{S^2}d\omega
b(\xi-\xi_*,\omega)ff_*=fL(f),$$
with
$$L(f)=A*_{\xi}f,\qquad A(z)=\int_{S^2}b(z,\omega)d\omega,\quad z\in\R^3.$$

The  collision kernel $b$ in the collision operator $Q$ is
a given function on $\R^3\times S^2$. We shall always assume  the so-called {\textit{angular cut-off}} kernel throughout
the rest of the paper, that is, $b$ satisfies
$$b\in L^1(B_R\times S^2)\quad \textrm{for all}\quad R\in
(0,\infty),\qquad b\geq 0$$ where $B_R=\{z\in \R^3: |z|<R\}$, and
\begin{equation*}
\begin{cases}
b(z,w)\quad \textrm{depends only on}\quad |z|\quad
\textrm{and}\quad |(z,\omega)|,\\
(1+|z|^2)^{-1}\left(\int_{z+B_R}A(\xi)d\xi\right)\rightarrow
0,\quad \textrm{as}\quad |z|\rightarrow\infty,\quad \textrm{for
all}\quad R\in(0,\infty).
\end{cases}
\end{equation*}
A classical example of such angular cut-off collision kernels is given by the so-called hard-spheres model
where we have
$$b(z,\omega)=|(z,\omega)|.$$

The VMB system \eqref{21} is complemented with the initial conditions
\begin{equation}\label{VMBic}
\begin{cases}
f|_{t=0}=f_0,\quad\textrm{on}\quad \R^6,\quad\textrm{with}\quad f_0\ge 0,\\
E|_{t=0}=E_0,\quad B|_{t=0}=B_0\quad\textrm{on} \quad\R^3_x,
\end{cases}
\end{equation}
with the usual compatibility condition
$$\Dv B_0=0,\qquad\textrm{and}\qquad\Dv E_0=\r_0=\int_{\R^3}f_0d\xi,\quad\textrm{on}\quad \R^3_x.$$

We state below our main stability results concerning the Cauchy problem of the
Vlasov-Maxwell-Boltzmann system \eqref{21} and \eqref{VMBic}. We assume that $f_0$ satisfies
\begin{equation}\label{IC}
\int_{\R^6}f_0(1+\nu+|\xi|^2+|\log
f_0|)dxd\xi+\int_{\R^3}(|E_0|^2+|B_0|^2)dx<\infty,
\end{equation}
where $\nu=\nu(x)$ is some function in $\R^3$ satisfying
$$\nu\ge 0,\quad (1+\nu)^{\f{1}{2}}\quad\textrm{is Lipschitz
on}\quad\R^3,\quad e^{-\nu}\in L^1(\R^3).$$

Using the classical identity (see Lemma 2.1 in \cite{FFM}),
\begin{equation}\label{210}
\int_{\R^3}Q(f,f)\zeta(\xi)d\xi=\f{1}{4}\int_{\R^6}dxd\xi_*\int_{S^2}d\omega
b(f'f'_*-ff_*)\left(\zeta+\zeta_*-\zeta'-\zeta'_*\right),
\end{equation}
we deduce the following local conservation laws of mass, momentum and
kinetic energy:
\begin{equation}\label{22}
\f{\partial \r}{\partial t}+\Dv_xj=0,
\end{equation}
\begin{equation}\label{23}
\begin{split}
&\f{\partial }{\partial t}\left(\int_{\R^3}f\xi d\xi+E\times B\right)\\
&\qquad +\Dv_x\left(\int_{\R^3}\xi\otimes\xi
fd\xi+\left(\f{|E|^2+|B|^2}{2}Id-E\otimes E-B\otimes B\right)
\right)=0,
\end{split}
\end{equation}
\begin{equation}\label{24}
\f{\partial}{\partial
t}\left(\int_{\R^3}f|\xi|^2d\xi\right)+\Dv_x\left(\int_{\R^3}\xi|\xi|^2
fd\xi\right)-2E\cdot\int_{\R^3}\xi fd\xi=0,
\end{equation}
for $(x,t)\in\R^3\times(0,\infty)$.
In fact, while \eqref{22} and \eqref{24} are easy to verify, we
need to pay more attention to \eqref{23}. To verify \eqref{23}, we
first multiply \eqref{e211} by $\xi$ and integrate with respect to
$\xi$ to obtain
\begin{equation}\label{2101}
\f{\partial}{\partial t}\int_{\R^3}f\xi
d\xi+\Dv_x\int_{\R^3}\xi\otimes\xi fd\xi=-\left(\r E+j\times
B\right).
\end{equation}
Note that
$$E\Dv E+(\nabla\times E)\times E=\Dv(E\otimes E)-\f{1}{2}\nabla|E|^2.$$
Thus it yields the following, combined
with \eqref{e212} and \eqref{e213},
\begin{equation}\label{2100}
\f{\partial}{\partial t}(E\times
B)+\Dv_x\left(\f{|E|^2+|B|^2}{2}Id-E\otimes E-B\otimes
B\right)=-(\r E+j\times B).
\end{equation}
Then, adding \eqref{2101} and \eqref{2100} together gives
\eqref{23}. Integrating \eqref{22}-\eqref{24} in $x$
over $\R^3$, we deduce the following global conservation of mass,
momentum and total energy
\begin{equation}\label{25}
\f{d }{d t}\int_{\R^6}fdxd\xi=0,\quad\textrm{for}\quad t\ge 0,
\end{equation}
\begin{equation}\label{26}
\f{d}{d t}\left(\int_{\R^6}f\xi dxd\xi+\int_{\R^3}E\times B
dx\right) =0,
\end{equation}
\begin{equation}\label{27}
\f{d}{d
t}\int_{\R^6}f|\xi|^2dxd\xi-2\int_{\R^3}E\cdot\int_{\R^3}\xi fd\xi
dx=0,\quad\textrm{for}\quad t\ge 0.
\end{equation}
On the other hand, multiplying \eqref{e212} by $E$, multiplying
\eqref{e213} by $B$, integrating them in $x$ over $\R^3$,  and then
summing them together, we obtain
$$\f{d}{dt}\int_{\R^3}(|E|^2+|B|^2)dx=-2\int_{\R^3}E\cdot j dx.$$
Substituting the above identity back to \eqref{27}, one obtains
\begin{equation}\label{28}
\f{d}{d
t}\left(\int_{\R^6}f|\xi|^2dxd\xi+\int_{\R^3}(|E|^2+|B|^2)dx\right)=0,\quad\textrm{for}\quad
t\ge 0.
\end{equation}

Therefore, if we assume that the initial condition $f_0$ as
\eqref{IC}, we deduce from \eqref{25}, \eqref{26} and \eqref{28}
that
\begin{equation}\label{29}
\sup_{t\in
[0,T]}\int_{\R^6}f(1+\nu+|\xi|^2)dxd\xi+\int_{\R^3}(|E|^2+|B|^2)dx\le
C(T)
\end{equation}
for some nonnegative constant $C(T)$ that depends only on $T$ and
on the initial data. Indeed, we observe that we have, multiplying
\eqref{e211} by $\nu(x)$ and then integrating over $\xi$,
\begin{equation*}
\begin{split}
\f{\partial}{\partial t
}\left(\int_{\R^3}f\nu(x)d\xi\right)+\Dv_x\left(\int_{\R^3}f\nu(x)\xi
d\xi\right)&=\int_{\R^3}f\xi\cdot\nabla_x\nu(x)d\xi\\
&\le\f{1}{2}\int_{\R^3}f|\xi|^2d\xi+\f{1}{2}\r(t,x)|\nabla\nu|^2\\
&\le
\f{1}{2}\int_{\R^3}f|\xi|^2d\xi+C\int_{\R^3}fd\xi+C\int_{\R^3}f\nu
d\xi,
\end{split}
\end{equation*}
since $(1+\nu)^{\f{1}{2}}$ is Lipschitz. In particular, we deduce
$$\f{d}{dt}\int_{\R^6}f\nu(x)dxd\xi\le
C+\f{1}{2}\int_{\R^6}f(|\xi|^2+\nu(x))dxd\xi.$$ Then \eqref{29}
follows from the above inequality and Gr\"{o}nwall's inequality.

The final formal bound we wish to obtain is deduced from the
entropy identity. Multiplying \eqref{e211} by $\log f$, using
\eqref{210}, we obtain, at least formally,
\begin{equation}\label{212}
\begin{split}
\f{d}{dt}\int_{\R^6}f\log
fdxd\xi+\f{1}{4}\int_{\R^3}dx\int_{\R^6}d\xi d\xi_*\int_{S^2}
B(f'f'_*-ff_*)\log\f{f'f'_*}{ff_*}=0.
\end{split}
\end{equation}
Since the second term is clearly nonnegative, we deduce in
particular that
\begin{equation}\label{211}
\sup_{t\ge 0}\int_{\R^6}f\log fdxd\xi\le \int_{\R^6}f_0\log
f_0dxd\xi.
\end{equation}
This inequality together with a lemma in \cite{DL2} implies
$$\sup_{t\in [0,T]}\int_{\R^6}f|\log f|\le C(T).$$
Also, if we go back to \eqref{212}, we   deduce that,
$$\int_0^Tdt\int_{\R^3}dx\int_{\R^6}d\xi d\xi_*\int_{S^2}
b(f'f'_*-ff_*)\log\f{f'f'_*}{ff_*}\le C(T).$$ In conclusion, we
obtain the following bounds:
\begin{equation}\label{ee}
\begin{split}
&\sup_{t\in [0,T]}\left(\int_{\R^6}f(1+|\xi|^2+\nu(x)+|\log
f|)dxd\xi+\int_{\R^3}(|E|^2+|B|^2)dx\right)\le C(T); \\
&\int_0^T\int_{\R^3}dx\int_{\R^6}d\xi d\xi_*\int_{S^2}d\omega
b(f'f'_*-ff_*)\log \f{f'f'_*}{ff_*}\le C(T).
\end{split}
\end{equation}
\bigskip

Now we give the definition of $\textit{Renormalized Solutions}$ to
VMB.
\begin{Definition}
A triple $(f(t, x, \xi), E(t, x), B(t, x))$ with $f\ge 0$ is said
to be a renormalized solution to VMB \eqref{21} if for all $T\in (0,\infty)$,
we have
\begin{itemize}
\item $f(t,x,\xi)\in
C([0,T]; L^1(\R^6))$, $E, B\in C([0,T]; L^2(\R^3))$ and
\eqref{ee} holds;\\
\item for any $\beta\in C^1([0,\infty))$ satisfying that
$\beta(0)=0$ and $\beta'(t)(1+t)$ is bounded in $[0,\infty)$,
\begin{equation}\label{df}
\f{\partial}{\partial
t}\beta(f)+\xi\cdot\nabla_x\beta(f)+(E+\xi\times B)\cdot\nabla_\xi
\beta(f)=\beta'(f)Q(f,f)
\end{equation}
holds in $\mathcal{D}'$ (sense of distributions); and\\
\item \eqref{e212} and \eqref{e213} hold in $\mathcal{D}'$.
\end{itemize}
\end{Definition}


One of the main objectives in the rest of this paper is devoted to
the stability of renormalized solutions to VMB. More precisely, we
consider a sequence of  initial data $\{(f_0^n, E_0^n, B_0^n)\}_{n=1}^\infty$ satisfying \eqref{IC} with $f_0^n\ge
0$, $a.e.$ in $\R^6$ and converging to $(f_0, E_0, B_0)$. Then, corresponding to those initial
conditions, we suppose that there is a sequence of renormalized
solutions $\{(f^n, E^n, B^n)\}_{n=1}^\infty$ to VMB satisfying
\eqref{ee}. Without loss of generality, we may assume that $(f^n,
E^n, B^n)$ converges weakly to $(f, E, B)$. We will prove that
$(f, E, B)$ is still a renormalized solution to VMB with the
initial data $(f_0, E_0, B_0)$.

\begin{Theorem}[Weak Stability] \label{2T1}
Suppose that $\{(f^n, E^n, B^n)\}_{n=1}^\infty$ is a sequence of renormalized
solutions to VMB \eqref{21} satisfying \eqref{ee},
with initial data $\{(f_0^n, E_0^n, B_0^n)\}_{n=1}^\infty$ satisfying \eqref{IC},  $f_0^n\ge 0$, $a.e.$ in $\R^6$
 and converging weakly to $(f_0, E_0, B_0)$ in $L^1(\R^6)\times\left(L^2(\R^3)\right)^6$; and $(f, E, B)$ is a weak-$\ast$ limit of $\{(f^n, E^n, B^n)\}$ in $L^\infty(0,T;L^1(\R^6))\times\left(L^\infty(0,T;\left(L^2(\R^3)\right)^6\right)$.
Then the sequence $\{f_n\}$ satisfies:
\begin{enumerate}
\item For all $\psi\in C(\R^3)$ such that $\f{|\psi(\xi)|}{1+|\xi|^2}\rightarrow 0$ as $|\xi|\rightarrow \infty$, $\int_{\R^3}f^n\phi d\xi$ converges to
$\int_{\R^3}f\psi d\xi$ in $L^p([0,T], L^1_{loc}(\R^3))$ for all
$1\le p<\infty$.
\item $L(f^n)$ converges to
$L(f)$ in $L^p([0,T]; L^1(\R^3_x\times K))$ for all $1\le
p<\infty$, $T\in (0,\infty)$, $K$ compact set in $\R^3_\xi$.
\item For all $\phi\in L^\infty(\R^3)$ with compact support,
$\int_{\R^3}Q^{\pm}(f^n,f^n)\phi d\xi$ converges locally in
measure to $\int_{\R^3}Q^{\pm}(f,f)\phi d\xi$. And
$Q^{\pm}(f^n,f^n)(1+f^n)^{-1}$ are relatively weakly compact in
$L^1(\R_x^3\times K\times (0,T))$ for all $T\in(0,\infty)$,  compact set $K$ in $\R_{\xi}^3$.
\item $Q^+(f^n,f^n)$ converges
locally in measure to $Q^+(f,f)$.
\end{enumerate}
Moreover, if
$$\|E^n\|_{L^\infty(0,T; L^5(\R^3))}\quad \textrm{is uniformly
bounded},$$ and
$$\|B^n\|_{L^\infty(0,T; L^s(\R^3))}\quad \textrm{is uniformly
bounded for some }s>5,$$ then the weak limit $(f, E, B)$ is a
renormalized solution of \eqref{21} with the initial data $(f_0, E_0, B_0)$.
\end{Theorem}

\begin{Remark}\label{R1}
Due to the convexity of $x\ln x$ and the monotonicity of
$(x-y)\ln\f{x}{y}$ for all $x,y>0$, we can show, as in \cite{DL5},
$$\int_{\R^6}f(t)\ln f(t)dxd\xi\le\liminf_{n\rightarrow
\infty}\int_{\R^6}f^n(t)\ln f^n(t)dxd\xi,$$ and
\begin{equation*}
\begin{split}
&\int_0^t\int_{\R^3}\int_{\R^6}d\xi d\xi_*\int_{S^2}d\omega
b(f'f'_*-ff_*)\ln \f{f'f'_*}{ff_*}\\
&\quad\le \liminf_{n\rightarrow
\infty}\int_0^t\int_{\R^3}\int_{\R^6}d\xi d\xi_*\int_{S^2}d\omega
b({f^n}'{f^n}'_*-f^nf^n_*)\ln \f{{f^n}'{f^n}'_*}{f^nf^n_*},
\end{split}
\end{equation*}
for all $t\ge 0$. This entropy estimate is crucial for the long
time behavior of renormalized solutions.
\end{Remark}

A consequence of the weak stability is the propagation of
smoothness of renormalized solutions.

\begin{Theorem}[Propagation of Smoothness] \label{2T2}
If,  in addition to the assumptions in Theorem \ref{2T1},  $f_0^n$
converges in $L^1(\R^6)$ to $f_0$, then $f^n$ converges to $f$ in
$C([0,T]; L^1(\R^6))$ for all $T\in [0,\infty)$, and $(f, E, B)$
is a renormalized solution of \eqref{21} if $(f^n, E^n, B^n)$ is a sequence of renormalized
solutions.
\end{Theorem}

\begin{Remark}
The assumption that $E(x,t)$ is uniformly bounded in
$L^\infty(0,T;L^5(\R^3))$ is crucial for Theorems \ref{2T1} and
\ref{2T2}, because of the nonlinear term associated with the Lorentz
force. Notice that usually from Maxwell's equations, we can only
obtain the $\textit{a priori}$ estimates on $E$ and $B$ in
$L^\infty(0,T; L^2(\R^3))$.
\end{Remark}

\bigskip

\section{Proof of Theorem \ref{2T1}: Weak Stability}

This section is devoted to the proof of  Theorem \ref{2T1}. We divide the proof into two steps. In the first step we
show why the first four statements of  the theorem hold.  Then we concentrate in the second
step on the proof of the fact that the weak limit is indeed a
renormalized solution of Vlasov-Maxwell-Boltzmann equations. We
remark that the first step is essentially an adaptation of the
results and methods of \cite{DL4, DL2, DL3}, while the second one
requires a new result of renormalized solutions for the Vlasov-Maxwell
equations.

\subsection{Step One}
In this subsection, we are aiming at proving the first statement of Theorem \ref{2T1}
following the spirit of \cite{DL4}. Then the second and the
third statements can be  shown exactly as in \cite{DL4}. Finally, once the first
three statements hold, the fourth statement will immediately
follows from the argument in \cite{DL2}.
Therefore, for the sake of conciseness, we only give the detailed proof of the first statement of Theorem \ref{2T1}.

In order to prove the first statement, we first recall that for all compact sets
$K\subset \R^3_{\xi}$ and  $T\in (0,\infty)$, we have
\begin{equation}\label{4100}
\begin{split}
\int_{\R^3\times K} (1+f^n)^{-1}Q^{-}(f^n,f^n)dx d\xi
&\le \int_{\R^3\times K}  L(f^n) dx d\xi \\
&=\int_{\R^3}dx\int_{\R^3}f^n(x,\xi_*, t)\int_K
A(\xi-\xi_*)d\xi d\xi_*\\
&\le C\int_{\R^3}dx\int_{\R^3}
f^n(x,\xi_*,t)(1+|\xi_*|^2)d\xi_*<\infty,
\end{split}
\end{equation}
due to the assumption on the collision kernel $b$, hence,
$$(1+f^n)^{-1}Q^{-}(f^n,f^n)\quad\textrm{is bounded in}\quad
L^\infty(0,T; L^1(\R^3\times K)).$$ Also, we observe that we have (see \cite{DL4}),
\begin{equation*}
\begin{split}
Q^{+}(f^n,f^n)\le 2Q^{-}(f^n,f^n)+\f{1}{\ln
2}\int_{\R^3}d\xi_*\int_{S^2} b d\omega
(f^{n'}-f^n_*-f^nf^n_*)\ln\f{f^{n'}f^{n'}_*}{f^nf^n_*},
\end{split}
\end{equation*}
which, combining with \eqref{ee} and \eqref{4100}, implies
\begin{equation}\label{4101}
(1+f^n)^{-1}Q^{+}(f^n,f^n)\quad\textrm{is bounded in}\quad
L^1(0,T; L^1(\R^3\times K))
\end{equation}
for all compact sets $K$ in $\R_\xi^3$ and $T\in (0,\infty)$.

Next, we observe that since $f^n$ is a renormalized solution of
VBM \eqref{21}, we have, for $\beta=\beta_\dl=\f{t}{1+\dl t}$,
\begin{equation}\label{32}
\left(\f{\partial}{\partial
t}+\xi\cdot\nabla_x\right)\beta_\dl(f^n)=\beta_\dl'(f^n)Q(f^n,f^n)-\Dv_\xi((E^n+\xi\times
B^n)\beta_\dl(f^n))
\end{equation}
in $\mathcal{D}'$. In order to apply the velocity averaging
results in \cite{DL1, DLM}, we remark that \eqref{4100} and
\eqref{4101} imply that $\beta_\dl'(f^n)Q(f^n,f^n)$ is bounded in
$L^1(0,T; L^1(\R_x^3\times K))$ for all compact subsets $K$ of
$\R_\xi^3$. And also we observe that $\beta_\dl(f^n)$ is bounded
in $L^\infty((0,T)\times\R^6))$, and hence,
$\Dv_\xi((E^n+\xi\times B^n)\beta_\dl(f^n))$ is bounded in
$L^2((0,T)\times \R^3; H^{-1}_{\xi}(\R^3))$.
Denoting
$$\mathcal{T}_\dl(f^n)=\beta_\dl'(f^n)Q(f^n,f^n),$$
 and decomposing $\beta_\dl(f^n)$ into
$$u^n=\beta_\dl(f^n)\chi_{\{(t,x,\xi):|\mathcal{T}_\dl(f^n)|\le
M\}}\in L^2((0,T)\times \R^6)),$$
 $g^n$, and $h^n$ by
\begin{equation}\label{fp1}
\begin{split}
\left(\f{\partial}{\partial t}+\xi\cdot\nabla_x\right)u^n=&\mathcal{T}_\dl(f^n)\chi_{\{(t,x,\xi):|\mathcal{T}_\dl(f^n)|\le M\}}\\
&-\Dv_\xi\left((E^n+\xi\times B^n)\beta_\dl(f^n)\chi_{\{(t,x,\xi): |\mathcal{T}_\dl(f^n)|\le M\}}\right),
\end{split}
\end{equation}
\begin{equation}\label{fp2}
\left(\f{\partial}{\partial
t}+\xi\cdot\nabla_x\right)g^n=-\Dv_\xi((E^n+\xi\times
B^n)\beta_\dl(f^n)\chi_{\{(t,x,\xi):|\mathcal{T}_\dl(f^n)|> M\}}),
\end{equation}
\begin{equation}\label{fp3}
\left(\f{\partial}{\partial
t}+\xi\cdot\nabla_x\right)h^n=\mathcal{T}_\dl(f^n)\chi_{\{(t,x,\xi):|\mathcal{T}_\dl(f^n)|>
M\}},
\end{equation}
for $M>1$, where
$$h^n|_{t=0}=g^n|_{t=0}=0, \quad
u^n|_{t=0}=\beta_\dl(f^n_0),$$
and $\chi$ is the characteristic function of sets. Because
$\{\mathcal{T}_\dl(f^n)\}_{n=1}^\infty$ is weakly compact in
$L^1((0,T)\times\R^6)$ due to the facts that
$\beta'(t)=\f{1}{(1+\dl t)^2}\le\f{1}{1+\dl t}$ and
$\f{1}{1+f^n}Q(f^n, f^n)$ is weakly compact in
$L^1((0,T)\times\R^6)$, and because, from \eqref{fp3},
\begin{equation*}
h^n(t, x+t\xi,
\xi)=\int_0^t\mathcal{T}_\dl(f^n)\chi_{\{(t,x,\xi):|\mathcal{T}_\dl(f^n)|\ge
M\}}(\tau, x+\xi\tau,\xi)d\tau,
\end{equation*}
it follows that,  uniformly with respect to $n$,
\begin{equation}\label{al1}
\int_0^T\int_{\R^3}\int_{\R^3}|h^n(t, x, \xi)|d\xi dxdt\rightarrow
0,\quad\textrm{as}\quad M\rightarrow \infty.
\end{equation}
Similarly, from the compactness of $\mathcal{T}_\dl(f^n)$, we
deduce that
\begin{equation}\label{s}
\mathcal{S}^n:=(E^n+\xi\times
B^n)\beta_\dl(f^n)\chi_{\{(t,x,\xi):|\mathcal{T}_\dl(f^n)|>
M\}}\rightarrow 0
\end{equation}
in $L_{loc}^1((0,T)\times\R^6)$ as $M\rightarrow \infty $. From
\eqref{fp2}, we have
\begin{equation*}
g^n(t, x+t\xi, \xi)=\int_0^t-\Dv_\xi((E^n+\xi\times
B^n)\beta_\dl(f^n)\chi_{\{(t,x,\xi):|\mathcal{T}_\dl(f^n)|>
M\}})(\tau, x+\xi\tau,\xi)d\tau.
\end{equation*}
Thus, for any $\psi\in\mathcal{D}_\xi(\R^3)$, we deduce from the
above identity that
\begin{equation*}
\begin{split}
&\int_{\R^3}g^n(t, x+t\xi,
\xi)\psi(\xi)d\xi\\&\quad=\int_0^t\int_{\R^3}((E^n+\xi\times
B^n)\beta_\dl(f^n)\chi_{\{(t,x,\xi):|\mathcal{T}_\dl(f^n)|>
M\}})(\tau, x+\xi\tau,\xi)\cdot\nabla_\xi\psi d\xi d\tau.
\end{split}
\end{equation*}
Therefore, from the weak compactness of $\mathcal{S}^n$, the above
identity with \eqref{s} implies
\begin{equation}\label{s1}
\int_{\R^3}g^n(t, x, \xi)\psi d\xi\rightarrow
0,\quad\textrm{as}\quad M\rightarrow \infty,
\end{equation}
in $L^1_{loc}((0,T)\times \R^3)$.

 On the other hand, since $\{u^n\}_{n=1}^\infty$ and
$\left\{\mathcal{T}_\dl(f^n)\chi_{\{(t,x,\xi):|\mathcal{T}_\dl(f^n)|\le
M\}}\right\}_{n=1}^\infty$ are bounded sequences in
$L^2((0,T)\times \R^6)$, and $\Dv_\xi((E^n+\xi\times
B^n)\beta_\dl(f^n))$ is bounded in $L^2((0,T)\times \R^3;
H^{-1}_{\xi}(\R^3))$, by the velocity averaging lemma (Theorem 3
in \cite{DL1}), we deduce that
\begin{equation*}
\int_{\R^3}u^n \psi(\xi)d\xi\quad\textrm{is bounded in}\quad
H^{\f{1}{4}}((0,T)\times \R^3),
\end{equation*}
for all $\psi\in\mathcal{D}(\R^3)$. Thus, $\left\{\int_{\R^3}u^n
\psi(\xi)d\xi\right\}_{n=1}^\infty$ is compact in $L^2((0,T)\times\R^3)$ and is locally compact in
$L^1((0,T)\times\R^3)$,  which,  combining with \eqref{al1} and
\eqref{s1}, implies that
\begin{equation}\label{33}
\int_{\R_\xi^3}\beta_\dl(f^n)\psi d\xi\quad\textrm{is relatively
compact in}\quad L^p(\times(0,T), L_{loc}^1(\R^3))
\end{equation}
for all $1\le p<\infty$, $\psi\in \mathcal{D}(\R^3)$.

The first statement  of the theorem  for $\psi\in\mathcal{D}(\R^3)$ then follows
from \eqref{33} and \eqref{ee}, since it suffices to observe that
we have for all $R>1$,
\begin{equation}\label{*}
0\le f^n-\beta_\dl(f^n)\le \dl Rf^n+f^n\chi_{\{f^n>R\}}\le \dl
Rf^n+f^n\f{\ln f^n}{\ln R},
\end{equation}
and then take the limit as $R\rightarrow\infty$ and
$\dl\rightarrow 0$.
Next, for a general $\psi\in C(\R^3)$ such that
$\psi(\xi)(1+|\xi|^2)^{-1}\rightarrow 0$ as $|\xi|\rightarrow
\infty$, we introduce $$\eta_M=\eta\left(\f{\cdot}{M}\right),$$  for
$M>1$, where $\eta\in \mathcal{D}(\R^3)$, $0\le\eta\le 1$,
$\eta=1$ on $B_1$. Then the first statement holds for
$\psi\eta_M$, and the first statement will be valid for such a
$\psi$ provided  that
\begin{equation} \label{L1}
\sup_{n}\int_0^Tdt\int_{K\times\R^3}f^n|\psi|(1-\eta_M)dxd\xi\rightarrow
0\quad\textrm{as}\quad M\rightarrow\infty,
\end{equation}
for compact subsets $K\in \R_x^3$. Indeed, \eqref{L1} follows from \eqref{ee} since
\begin{equation*}
\begin{split}
&\int_0^T\int_{K\times\R^3}f^n|\psi|(1-\eta_M)dxd\xi\\
&\quad\le C\sup_{|\xi|\ge
M}\f{|\psi(\xi)|}{1+|\xi|^2}\int_0^Tdt\int_{K\times\R^3}f^n(1+|\xi|^2)\chi_{\{|\xi|\ge
M\}}\\
&\quad\le C\sup_{|\xi|\ge M}\f{|\phi(\xi)|}{1+|\xi|^2}
\end{split}
\end{equation*}
for some $C>0$ independent of $n$.

\subsection{Step Two}
We now aim at proving that $(f, E, B)$ is a renormalized solution
of VBM. First of all, we claim that it is enough to show that
\begin{Lemma}\label{2l1}If $f\in L^\infty(0,T; L^1(\R^6))$, the equation \eqref{df} holds if and only if
\begin{equation}\label{35}
\f{\partial}{\partial t}\ln(1+f)+\Dv_x(\xi \ln(1+f))+(E+\xi\times
B)\cdot\nabla_\xi \ln(1+f)=\f{1}{1+f}Q(f,f),
\end{equation}
in $\mathcal{D}'$.
\end{Lemma}

\begin{proof}
On one hand, if $f$ is a renormalized solution to VMB, then
\eqref{35} automatically holds since $\beta(f)=\ln(1+f)\in
C^1([0,\infty))$ with $\beta(0)=0$ and $\beta'(f)(1+f)=1$.

On the other hand, if \eqref{35} holds, we claim that $f$ is a
renormalized solution to VMB. Indeed, denoting
$$\sigma(s)=\beta(e^s-1)$$ for all
$\beta(t)\in C^1([0,\infty))$ with $\beta(0)=0$ and
$\beta'(t)(1+t)\le C$. Then, we have
$$\partial_t\sigma(f)=\sigma'(f)\partial_t f;$$
$$\nabla_x\sigma(f)=\sigma'(f)\nabla_x f;$$
$$\nabla_\xi\sigma(f)=\sigma'(f)\nabla_\xi f;$$
Multiplying \eqref{2l1} by $\sigma'(\ln(1+f))$, we obtain,
\begin{equation}\label{35111111111111111111}
\begin{split}
&\f{\partial}{\partial t}\sigma(\ln(1+f))+\Dv_x(\xi
\sigma(\ln(1+f)))+(E+\xi\times B)\cdot\nabla_\xi \sigma(\ln(1+f))\\
&=\sigma'(\ln(1+f))\f{1}{1+f}Q(f,f),
\end{split}
\end{equation}
in the sense of distributions. Note that, by the definition
$\sigma$, we have
$$\sigma(\ln(1+f))=\beta(f),\quad\textrm{and}\quad
\sigma'(\ln(1+f))=\beta'(f)(1+f).$$ Hence, substituting the  above two
identities in \eqref{35111111111111111111}, we get
\begin{equation*}
\f{\partial}{\partial t}\beta(f)+\Dv_x(\xi \beta(f))+(E+\xi\times
B)\cdot\nabla_\xi \beta(f)=\beta'(f)Q(f,f),
\end{equation*}
in the sense of distributions.
The proof of this lemma is complete.
\end{proof}

The rest of this subsection is devoted to the proof of \eqref{35}.
Recall that we deduce,  from $\textit{a priori}$ estimate \eqref{ee}
and weak passages to the limit,
\begin{equation}\label{36}
\begin{split}
\sup_{t\in [0,T]}&\left(\int_{\R^6}f(1+|\xi|^2+\nu(x)+|\log
f|)dxd\xi+\int_{\R^3}(|E|^2+|B|^2)dx\right)\\
&+\int_0^T\int_{\R^3}dx\int_{\R^6}d\xi d\xi_*\int_{S^2}d\omega
b(f'f'_*-ff_*)\log \f{f'f'_*}{ff_*}<\infty,
\end{split}
\end{equation}
for all $T\in (0,\infty)$.  Now  the strategy   to
prove \eqref{35} is the following: we first consider
$$\beta_\dl(f^n)=f^n(1+\dl f^n)^{-1}$$ for $\dl\in (0,1]$ and weakly
pass to the limit as $n$ goes to $\infty$ in the equation
satisfied by $\beta_\dl(f^n)$; then for the  equation satisfied by the limit of $\beta_\dl(f^n)$ as $n\rightarrow \infty$,  we use $\beta$ to renormalize it and let $\dl$ go to $0$ to recover \eqref{35}. 
To begin with, without loss of
generality, in view of \eqref{ee}, we can assume
$$f^n\rightarrow f\qquad\textrm{weakly}^*\quad \textrm{in}\quad
L^\infty(0,T; L^1(\R^6));$$
$$B^n\rightarrow B\qquad\textrm{weakly}^*\quad \textrm{in}\quad
L^\infty(0,T; L^2(\R^3));$$
$$E^n\rightarrow E\qquad\textrm{weakly}^*\quad \textrm{in}\quad
L^\infty(0,T; L^2(\R^3)\cap L^5(\R^3)).$$ Furthermore, without
loss of generality, extracting subsequence if necessary, we may
assume that for all $\dl>0$
\begin{equation}\label{37}
\beta_\dl(f^n)\rightarrow \beta_\dl\quad \textrm{weakly in}\quad
L^p(\R^6\times(0,T));
\end{equation}
\begin{equation}\label{38}
h_\dl^n=(1+\dl f^n)^{-2}\rightarrow h_\dl\quad\textrm{weakly-*
in}\quad L^\infty(\R^6\times(0,\infty));
\end{equation}
\begin{equation}\label{39}
g_\dl^n=f^n(1+\dl f^n)^{-2}\rightarrow g_\dl\quad\textrm{weakly
in}\quad L^p(\R^6\times(0,T)),
\end{equation}
for all $T\in (0,\infty)$, $1\le p\le \infty$. Furthermore,
because of the third statement and the equi-integrability, we may
assume that
\begin{equation}\label{310}
(1+\dl f^n)^{-2}Q^{\pm}(f^n,f^n)\rightarrow
Q^{\pm}_\dl\quad\textrm{weakly in}\quad L^1(\R^3_x\times K\times
(0,T)),
\end{equation}
for all compact sets $K\subset\R_\xi^3$ and $T\in(0,\infty)$.

Notice that, since $f^n$ is a renormalized solution of VMB,
\eqref{32} holds with $\beta(f^n)$ replaced by $\beta_\dl(f^n)$
for all $\dl>0$ and we want to pass to the limit in these
equations as $n$ goes to $\infty$. To this end, we deduce from the
first statement of Theorem \ref{2T1} that $\r^n$ and $j^n$
converge in $L^p(0,T; L^1(\R^3_x))$ to $\r$ and $j$, respectively
for all $1\le p<\infty$ and $T\in(0,\infty)$.
We then pass to the limit in \eqref{32} and we obtain
\begin{subequations}\label{311}
\begin{align}
&\f{\partial}{\partial t}\beta_\dl+\Dv_x(\xi
\beta_\dl)+(E+\xi\times
B)\cdot\nabla\beta_\dl=Q_\dl^+-Q_\dl^-,\quad x\in \R^3,\quad
\xi\in \R^3,\quad
t\ge 0,\label{e3111}\\
&\f{\partial E}{\partial t}-\nabla\times B=-j, \quad \Dv B=0,\quad
\textrm{on}\quad \R^3_x\times(0,\infty),\label{e3112}\\
&\f{\partial B}{\partial t}+\nabla\times E=0, \quad\Dv E=\r,\quad
\textrm{on}\quad \R^3_x\times(0,\infty),\label{e3113}\\
&\r=\int_{\R^3}fd\xi, \quad j=\int_{\R^3}f\xi d\xi, 
\quad \textrm{on}\quad
\R^3_x\times(0,\infty)\label{e3114},
\end{align}
\end{subequations}
in $\mathcal{D}'$. Here, for the convergence of the nonlinear term
$(E^n+\xi\times B^n)\cdot\nabla\beta_\dl(f^n)$, we need to show,
for all $\phi\in\mathcal{D}((0,\infty)\times\R^6)$,
\begin{equation}\label{ale}
\begin{split}
&\int_0^t\int_{\R^6}\phi(E^n+\xi\times
B^n)\cdot\nabla_\xi\beta_\dl(f^n)d\xi dxds\\
&=-\int_0^t\int_{\R^6}\nabla_\xi\phi\cdot(E^n+\xi\times
B^n)\beta_\dl(f^n) d\xi dxds,
\end{split}
\end{equation}
since $$\Dv_\xi(E^n+\xi\times B^n)=0.$$ If we take
$\phi=\ov{\phi}(t,x)\Phi(\xi)$ (which is enough by dense property)
for $\ov{\phi}\in\mathcal{D}((0,\infty)\times\R^3)$ and
$\Phi\in\mathcal{D}(\R^3)$, we can rewrite the term on the
right-hand side of \eqref{ale} as
$$-\int_0^t\int_{\R^3}\ov{\phi}(t,x)(E^n+\xi\times
B^n)\cdot\left(\int_{\R^3}\psi(\xi)\beta_\dl(f^n)d\xi\right)dxds,$$
by letting $\psi=\nabla_\xi \Phi.$ In fact, on one hand, by
\eqref{33} or the velocity averaging lemma in \cite {DLM},
$\int_{\R^3}\psi(\xi)\beta_\dl(f^n)d\xi$ and
$\int_{\R^3}\xi\psi(\xi)\beta_\dl(f^n)d\xi$ strongly converge to
$\int_{\R^3}\psi(\xi)\beta_\dl d\xi$ and
$\int_{\R^3}\xi\psi(\xi)\beta_\dl d\xi$ in $L^p(0,T;
L^1_{loc}(\R^3))$ respectively. On the other hand, since $\psi\in
C_0(\R^3)$ and $\beta_\dl(t)\le t$, we have, using \eqref{1116},
$$\left\{\int_{\R^3}\xi\psi
\beta_\dl(f^n)d\xi\right\}_{n=1}^\infty \quad\textrm{is uniformly
bounded in}\quad L^\infty(0,T; L^{\f{5}{4}}(\R^3)),$$ and
$$\left\{\int_{\R^3}\psi \beta_\dl(f^n)d\xi\right\}_{n=1}^\infty
\quad\textrm{is uniformly bounded in}\quad L^\infty(0,T;
L^{\f{5}{3}}(\R^3)).$$ The latter is true, because, for all $R>1$,
\begin{equation}\label{al3}
\begin{split}
\int_{\R^3}|\psi| \beta_\dl(f^n)d\xi&=\int_{\{|\xi|\le R\}}|\psi|
\beta_\dl(f^n)d\xi+\f{1}{R^2}\int_{\{|\xi|>R\}}|\xi|^2|\psi|
\beta_\dl(f^n)d\xi\\
&\le
\|\psi\|_{L^\infty}\left(R^3|B(0,1)|\|\beta_\dl(f^n)\|_{L^\infty}+\f{1}{R^2}\||\xi|^2\beta_\dl(f^n)\|_{
L_\xi^1(\R^3))}\right)\\
&\le
\|\psi\|_{L^\infty}\left(\f{R^3|B(0,1)|}{\dl}+\f{1}{R^2}\||\xi|^2f^n\|_{
L_\xi^1(\R^3))}\right),
\end{split}
\end{equation}
where $|B(0,1)|$ denotes the Lebesgue measure of the unit ball
$B(0,1)$ in $\R^3$, and by taking
$$R=\left(\f{\dl\||\xi|^2f^n\|_{
L_\xi^1(\R^3))}}{|B(0,1)|}\right)^{\f{1}{5}},$$ \eqref{al3}
becomes
\begin{equation*}
\int_{\R^3}|\psi| \beta_\dl(f^n)d\xi\le
2\|\psi\|_{L^\infty}\left(\f{|B(0,1)|^{\f{2}{5}}\||\xi|^2f^n\|_{
L_\xi^1(\R^3))}^{\f{3}{5}}}{\dl^{\f{2}{5}}}\right).
\end{equation*}
Therefore,
$$\left\{\int_{\R^3}|\psi|
\beta_\dl(f^n)d\xi\right\}_{n=1}^\infty\quad\textrm{is uniformly
bounded in}\quad L^\infty(0,T; L^{\f{5}{3}}_x(\R^3)),$$ since
$\{|\xi|^2f^n\}_{n=1}^\infty$ is uniformly bounded in
$L^\infty(0,T; L^1(\R^6))$. Thus $$\int_{\R^3}\xi\psi
\beta_\dl(f^n)d\xi\rightarrow \int_{\R^3}\xi\psi \beta_\dl d\xi
\quad\textrm{in}\quad L^p(0,T; L^s_{loc}(\R^3))\quad \textrm{for
all}\quad 1\le s<\f{5}{4},$$ and
$$\int_{\R^3}\psi \beta_\dl(f^n)d\xi\rightarrow \int_{\R^3}\psi \beta_\dl d\xi\quad\textrm{in}\quad L^p(0,T; L^r_{loc}(\R^3))\quad
\textrm{for all}\quad 1\le r<\f{5}{3},$$ for all $1\le p<\infty$.

The weak convergence of $E^n$ in $L^5((0,T)\times \R^3)$, combined
with the strong convergence of
$\int_{\R^3}\psi(\xi)\beta_\dl(f^n)d\xi$, implies
$$\int_0^t\int_{\R^6}\nabla_\xi\phi\cdot E^n \beta_\dl(f^n)
d\xi dxds\rightarrow \int_0^t\int_{\R^6}\nabla_\xi\phi\cdot E
\beta_\dl d\xi dxds.$$ The similar argument goes to the second
part of the nonlinear term
$$\int_0^t\int_{\R^3}\ov{\phi}(t,x)B^n\times
\left(\int_{\R^3}\psi(\xi)\xi\beta_\dl(f^n)d\xi\right)dxds,$$ due
to the weak convergence of $B^n$ in $L^q((0,T)\times \R^3)$ for
$q>5$. That is,
$$\int_0^t\int_{\R^6}\nabla_\xi\phi\cdot \xi\times
B^n\beta_\dl(f^n) d\xi dxds\rightarrow
\int_0^t\int_{\R^6}\nabla_\xi\phi\cdot \xi\times B\beta_\dl
d\xi.$$

Next, since $\beta_\dl(f^n)\in L^1(\R^6)\cap L^\infty(\R^6)$, we
know that $\beta_\dl\in L^\infty(\R^6)\cap L^1(\R^6)$. Also, since
$|\xi|^2f^n\in L^\infty(0,T; L^1(\R^6))$, we know that $\beta_\dl
(f^n)|\xi|^2\in L^\infty(0,T; L^1(\R^6))$ and
$\{\beta_\dl(f^n)\}_{n=1}^\infty$ is weakly compact in
$L^\infty(0,T; L^1(\R^6))$. Hence $|\xi|^2\beta_\dl\in
L^\infty(0,T; L^1(\R^6))$. Thus, for any $\sigma>0$, we have
$$\int_{\{\beta_\dl>\sigma\}}|\xi|^2dx d\xi<\f{1}{\sigma}\int_{\{\beta_\dl>\sigma\}}|\xi|^2\beta_\dl dx d\xi\le \f{1}{\sigma}\int_{\R^6}|\xi|^2\beta_\dl dx d\xi<\infty.$$
Therefore, Theorem \ref{1T1} implies that $\beta_\dl$ is a
renormalized solution of \eqref{311}.

As $\dl\rightarrow 0$, we claim that
\begin{Lemma}\label{2l2}
$$\beta_\dl\rightarrow f,\quad \textrm{in}\quad C([0,T]; L^1(\R^6)),$$
as $\dl\rightarrow 0$.
\end{Lemma}

\begin{proof} We start with proving the continuity of $\beta_\dl$ with respect to $t\ge 0$ with
values in $L^p(\R^6)$ for all $1\le p<\infty$.
To this end, we remark that if we regularize by convolution
$\beta_\dl$ into $\beta_{\dl}^{\i}$ as in Lemma \ref{1L1}, we
obtain
\begin{equation}\label{4000}
\f{\partial}{\partial
t}\beta_{\dl}^{\i}+\xi\cdot\nabla_x\beta_{\dl}^{\i}+(E+\xi\times
B)\cdot\nabla_\xi\beta_{\dl}^{\i}=Q^{+}_{\dl}-Q^{-}_{\dl}+r^{\i}
\end{equation}
where $r^{\i}\rightarrow 0$ in $L^1(0,T; L^1_{loc}(\R^6)\cap
L^\infty(\R^6))$ as $\i$ goes to $0$ for all $T\in(0,\infty)$.
Hence, it is easy to see from \eqref{4000} that,
$\beta_{\dl}^{\i}\in C([0,\infty); L^p(\R^6))$ for $1\le
p<\infty$. Note that $\beta_\dl$ is a renormalized solution to the
VM \eqref{e3111}. Subtracting \eqref{4000} from \eqref{311},
multiplying the result by
$|\beta_\dl-\beta_\dl^\i|^{p-2}(\beta_\dl-\beta_\dl^\i)$, and then
integrating over $\R^6$, we obtain
\begin{equation}\label{312}
\f{d}{dt}\int_{\R^6}|\beta_\dl-\beta_{\dl}^{\i}|^p
dxd\xi\rightarrow 0\quad \textrm{in} \quad L^1(0,T),\quad
\textrm{as}\quad \i\rightarrow 0
\end{equation}
for all $1\le p<\infty$, $T\in (0,\infty)$.  It follows that
$\beta_\dl\in C([0,T]; L^1(\R^6))$.

Next, we show that $f\in C([0,\infty); L^1(\R^6))$. Indeed,
because of \eqref{ee}, we have for all $T\in (0,\infty)$, as in
\eqref{*}
\begin{equation}\label{313}
\sup_{t\in [0,T]}\sup_{n\ge
1}\|f^n-\beta_\dl(f^n)\|_{L^1(\R^6)}\rightarrow 0\quad
\textrm{as}\quad \dl\rightarrow 0.
\end{equation}
Hence, by the lower semi-continuity of the weak convergence, we
obtain
\begin{equation*}
\begin{split}
\sup_{t\in [0,T]}\|f-\beta_\dl\|_{L^1(\R^6)}&\le \sup_{t\in
[0,T]}\liminf_{n\rightarrow
\infty}\|f^n-\beta_\dl(f^n)\|_{L^1(\R^6)}\\
&\le \sup_{t\in [0,T]}\sup_{n\ge
1}\|f^n-\beta_\dl(f^n)\|_{L^1(\R^6)}\rightarrow 0\quad
\textrm{as}\quad \dl\rightarrow 0,
\end{split}
\end{equation*}
and this implies $\beta_\dl$ converges in $C([0,T]; L^1(\R^6))$ to
$f$.
\end{proof}

Now we can state the equation \eqref{e3111} more precisely. To
this end, we observe that $\f{-t}{1+\dl t}$, $\f{1}{(1+\dl t)^2}$
are convex on $[0,\infty)$, therefore we have
\begin{equation}\label{314}
\beta_\dl\le\beta_\dl(f),\qquad h_\dl\ge (1+\dl f)^{-2}\quad
\textrm{a.e on} \quad \R^6\times(0,\infty).
\end{equation}
In addition $\f{t}{(1+\dl t)^2}=\beta_\dl(t)(1-\dl\beta_\dl(t))$,
because the function $x(1-\dl x)$ is a concave function, hence
\begin{equation}\label{315}
g_\dl\le \beta_\dl(1-\dl\beta_\dl)\quad\textrm{a.e on}\quad
\R^6\times(0,\infty).
\end{equation}
Furthermore, because of the second statement of Theorem \ref{2T1},
we deduce that
\begin{equation}\label{316}
Q^-_\dl=g_\dl L(f)\quad\textrm{a.e on}\quad \R^6\times(0,\infty).
\end{equation}
And, using the fourth statement of Theorem \ref{2T1}, we could
also deduce that
\begin{equation}\label{317}
Q_\dl^+=h_\dl Q^+(f,f)\quad\textrm{a.e on}\quad
\R^6\times(0,\infty).
\end{equation}

We finally use the fact that $\beta_\dl$ is a renormalized
solution of \eqref{311} to write
\begin{equation}\label{318}
\begin{split}
\f{\partial}{\partial t}\beta(\beta_\dl)&+\Dv_x(\xi\beta(\beta_\dl))+(E+\xi\times
B)\cdot\nabla_\xi \beta(\beta_\dl)\\
&=(1+\beta_\dl)^{-1}Q_\dl^+-(1+\beta_\dl)^{-1}Q_\dl^{-1}.
\end{split}
\end{equation}
And we wish to recover \eqref{35} by letting $\dl$ go to $0$.
Recall that we already showed in Lemma \ref{2l2} that $\beta_\dl$
converges to $f$ in $C([0,T]; L^1(\R^6))$ for all
$T\in(0,\infty)$. Therefore, in order to complete the proof of
Theorem \ref{2T1}, it only remains to show
\begin{Lemma}\label{2l3}
\begin{equation}\label{320}
Q^{\pm}_\dl(1+\beta_\dl)^{-1}\quad\textrm{are weakly relatively
compact in}\quad L^1(\R^3_x\times K\times(0,T))
\end{equation}
for all compact sets $K\subset \R^3_\xi$ and $T\in (0,\infty)$,
and
\begin{subequations}\label{319}
\begin{align}
&(1+\beta_\dl)^{-1}Q_{\dl}^{-}\rightarrow
(1+f)^{-1}Q^{-}(f,f),\quad\textrm{a.e}\label{3191}\\
&(1+\beta_\dl)^{-1}Q_{\dl}^{+}\rightarrow (1+f)^{-1}Q^{+}(f,f),
\quad\textrm{a.e}\label{3192}
\end{align}
\end{subequations}
as $\dl$ goes to $0$.
\end{Lemma}

\begin{proof}
We will follow the lines of the argument in \cite{DL3} and begin
with $Q_\dl^{-}$. Without loss of generality, we may assume that
$\beta_\dl$ converges $a.e.$ to $f$ as $\dl$ goes to $0$. Then,
\eqref{3191} follows since
$$(1+\beta_\dl)^{-1}Q^{-}_\dl=(1+\beta_\dl)^{-1}g_\dl
L(f)\rightarrow (1+f)^{-1}fL(f)$$ a.e. as $\dl\rightarrow 0$
provided we show that $g_\dl$ converges a.e. to $f$.

This is easy since we have for all $R>1$,
$$0\le f^n-f^n(1+\dl f^n)^{-2}\le 3R\dl f^n+f^n\chi_{\{f^n>R\}},$$
hence $g_\dl$ converges to $f$ in $C([0,T]; L^1(\R^6))$ for all
$T\in (0,\infty)$ by the uniform integrability of $f^n$ and the
lower semi-continuity of the weak convergence.  We now prove
\eqref{320} for $Q_{\dl}^{-}$ by first   observing  that
\eqref{316} yields
\begin{equation*}
\begin{split}
0&\le(1+\beta_\dl)^{-1}Q_\dl^{-}=(1+\beta_\dl)^{-1} g_\dl L(f)\\
&\le(1-\dl\beta_\dl)\f{\beta_\dl}{1+\beta_\dl} L(f)\le L(f),\quad
\textrm{a.e.}
\end{split}
\end{equation*}
And we conclude the proof of \eqref{320} for $Q_{\dl}^{-}$ by the equi-integrability, since $L(f)\in
L^\infty(0,T; L^1(\R^3_x\times K))$ for all compact sets $K\subset
\R^3_x$ and $T\in (0,\infty)$.

Next, we turn to the proof of  \eqref{320} for
$Q_\dl^{+}$ and \eqref{3192}. We begin with \eqref{320}.  We recall the
following classical inequality  for all $M> 1$,
\begin{equation}\label{321}
Q^+(f^n,f^n)\le MQ^{-}(f^n,f^n)+\f{1}{\ln M}\tilde{e}^n
\end{equation}
where
$$\tilde{e}^n=\int_{\R^3}d\xi_*\int_{S^2}  b \,  d\omega
({f^{n}}'{f^{n}_*}'-f^nf^n_*)\ln\f{{f^{n}}'{f^{n}_*}'}{f^nf^n_*}$$
is positive and bounded in $L^1(\R^6\times (0,T))$ for all $T\in
(0,\infty)$. Without loss of generality, we may assume that $\tilde{e}^n$
converges weakly in the sense of measures to some bounded
nonnegative measure $\tilde{e}$ on $\R^6\times[0,\infty)$ and we denote by
$\tilde{e}_0$ its regular part with respect to the usual Lebesgue measure,
that is, $\tilde{e}_0=\f{D\tilde{e}}{DyDt}$, $(y,t)\in\R^6\times (0,T)$. Dividing \eqref{321} by $(1+\dl
f^n)^2$ and letting $n$ go to $\infty$, we obtain
\begin{equation*}
Q_\dl^{+}\le M Q_\dl^{-}+\f{1}{\ln M}\tilde{e},
\end{equation*}
hence
\begin{equation*}
Q_\dl^{+}\le MQ_\dl^{-}+\f{1}{\ln M}\tilde{e}_0\quad\textrm{a.e. on} \quad
\R^6\times(0,\infty).
\end{equation*}
Then  \eqref{320}  for $Q_\dl^{+}$ follows since we already show it
for $Q_\dl^{-}$ and the integrability of $\tilde{e}_0$.

We finally prove \eqref{319} for $Q_\dl^{+}$. We first remark that
we have for all $R>0$,
\begin{equation}\label{324}
\begin{split}
Q^{+}(f^n,f^n)&\ge (1+\dl f^n)^{-2}Q^{+}(f^n,f^n)\\
&\ge (1+\dl R)^{-2}Q^{+}(f^n,f^n)\chi_{\{f^n<R\}}.
\end{split}
\end{equation}
In particular, if we multiply \eqref{324} by $\psi\in
C_0^\infty(\R_\xi^3)$ with $\psi\ge 0$, we find by letting $n$ go
to $\infty$ and using the third statement of Theorem \ref{2T1},
$$\int_{\R^3}Q^{+}(f,f)\psi d\xi\ge\int_{\R^3}Q^{+}_{\dl}\psi
d\xi\quad\textrm{a.e. on}\quad \R^3_x\times(0,\infty).$$
Indeed, the integrated left-hand side converges locally in measure while
the right-hand side converges weakly in $L^1$ and this is enough
to pass to the limit in the  inequality a.e. on
$\R^3_x\times(0,\infty)$. Therefore, we have for all $\dl\in (0,1]$,
\begin{equation}\label{325}
Q^{+}(f,f)\ge Q_\dl^{+}.
\end{equation}

Next, we use the other part of the inequality \eqref{324} and we
write for $\tau\in (0,1]$, using \eqref{321},
\begin{equation}\label{326}
\begin{split}
&(1+\dl R)^{-2}(1+\tau L(f^n))^{-1}Q^{+}(f^n,f^n)\\
&\le (1+\dl f^n)^{-2}Q^{+}(f^n,f^n)+(1+\tau
L(f^n))^{-1}\chi_{\{f^n>R\}}Q^{+}(f^n,f^n)\\
&\le (1+\dl f^n)^{-2}Q^{+}(f^n,f^n)+\f{1}{\ln
M}e^n+\f{M}{\tau}f^n\chi_{\{f^n>R\}}.
\end{split}
\end{equation}
We then observe that $Q^{+}(f^n,f^n)(1+\tau L(f^n))^{-1}$ is
relatively weakly compact in $L^1(\R^6\times(0,T))$ for all $T\in
(0,\infty)$ since it is bounded above by $\f{1}{\ln M}e^n+M\tau f^n$ for
all $M>1$. Hence, we may assume without loss of generality that it
converges weakly in $L^1(\R^6\times(0,T))$ for all
$T\in(0,\infty)$. We claim that its weak limit is given by $(1+\tau
L(f))^{-1}Q^{+}(f,f)$. Indeed, if $\psi\in L^\infty(\R^3_\xi)$
with compact support, we have
$$\int_{\R^3}(1+\tau L(f^n))^{-1}Q^{+}(f^n,f^n)\psi
d\xi=\int_{\R^3}Q^{+}(f^n,f^n)\psi_\tau^n d\xi,$$ where
$\psi_\tau^n$ is uniformly bounded in $L^\infty(\R^3_\xi)$, and has
a uniform compact support and $\psi_\tau^n\rightarrow
\psi_\tau=(1+\tau L(f))^{-1}\psi$ in $L^p((0,T)\times\R^6)$ for all
$1\le p<\infty$. This is enough to enable us to deduce
$$\int_{\R^3}Q^{+}(f^n,f^n)\psi_\tau^n d\xi\rightarrow
\int_{\R^3}Q^{+}(f,f)\psi_\tau d\xi$$ locally in measure on
$\R^3_x\times [0,\infty)$, which yields the claim.

We then pass to the limit in \eqref{326} and deduce as above
\begin{equation}\label{327}
\begin{split}
(1+\dl R)^{-2}(1+\tau L(f))^{-1}Q^{+}(f,f)\le Q^{+}_{\dl}+\f{1}{\ln
M}\tilde{e}_0+\f{M}{\tau}f_R,\quad \textrm{a.e.},
\end{split}
\end{equation}
where $f_R$ is the weak limit of $f^n\chi_{\{f^n>R\}}$ in
$L^1(\R^6)$. Since we have
$$\int_{\R^6}f_R dxd\xi =\lim_{n\rightarrow\infty}\int_{\R^6}
f^n\chi_{\{f^n>R\}}dxd\xi\le \f{C}{\ln R},$$ we deduce from
\eqref{327},  by letting first $\dl$ go to $0$, then $M$ go to
$\infty$, then $R$ go to $\infty$,  and finally $\tau$ go to $0$,  that
\begin{equation*}
Q^{+}(f,f)\le \lim_{\dl\rightarrow 0}Q^{+}_\dl\quad\textrm{a.e.},
\end{equation*}
which, combining with \eqref{325}, implies that
$$Q^{+}(f,f)= \lim_{\dl\rightarrow
0}Q^{+}_\dl\quad\textrm{a.e.}$$
The proof is complete.
\end{proof}

Putting together the conclusion of Step One, Lemmas \ref{2l1}-\ref{2l3}, we finish the proof of
Theorem \ref{2T1}.

\bigskip

\section{Proof of Theorem \ref{2T2}: Propagation of Smoothness}

In this section, we prove Theorem \ref{2T2}.
First, without loss of generality, in view of \eqref{ee}, we can
assume
$$f^n\rightarrow f\qquad\textrm{weakly}^*\quad \textrm{in}\quad
L^\infty(0,T; L^1(\R^6));$$
$$B^n\rightarrow B\qquad\textrm{weakly}^*\quad \textrm{in}\quad
L^\infty(0,T; L^2(\R^3)\cap L^5(\R^3));$$
$$E^n\rightarrow E\qquad\textrm{weakly}^*\quad \textrm{in}\quad
L^\infty(0,T; L^2(\R^3)\cap L^5(\R^3)).$$
Applying Theorem \ref{2T1}, we know that $f\in C([0,\infty);
L^1(\R^6))$ is a renormalized solution of VBM. In particular, we
know that we have, setting $\gamma_\dl(f)=\f{1}{\dl}\ln(1+\dl f)$,
\begin{equation}\label{41}
\begin{split}
\f{\partial}{\partial
t}\gamma_\dl(f)&+\Dv_x(\xi\gamma_\dl(f))+(E+\xi\times
B)\cdot\nabla_\xi\gamma_\dl(f)\\
&=\gamma'_\dl(f)Q^{+}(f,f)-f\gamma'_\dl(f)L(f),
\end{split}
\end{equation}
in $\mathcal{D}'$. It is easy to deduce that, $\gamma_\dl(f)\in
C([0,\infty); L^p(\R^6))$ for all $1\le p<\infty$ since
$$\gamma_\dl(f)\in C([0,\infty); L^1(\R^6))\cap L^\infty(0,\infty;
L^1(\R^6)), $$   hence
$$\gamma_\dl(f)|_{t=0}=\gamma_\dl(f_0)\quad\textrm{a.e. on}\quad
\R^6.$$

The strategy of the proof of Theorem \ref{2T2} goes as follows.
First of all, we introduce, without loss of generality, the weak
limit of $\gamma_\dl(f^n)$ in $L^p(\R^6\times(0,T))$ for all $T\in
(0,\infty)$ and $1\le p<\infty$, and we denote it by $\gamma_\dl$ (note the difference from the notation $\gamma_\dl(f^n)$ throughout this section).
The first step is to show that $\gamma_\dl$ is a
supersolution of \eqref{41}. In the second step, we deduce that
$\gamma_\dl=\gamma_\dl(f)$ and that $f^n$ converges to $f$ a.e. or
in $L^1(\R^6\times(0,T))$ for all $T\in(0,\infty)$. Finally in the
third step, we show that $f^n$ converges to $f$ in $C([0,T];
L^1(\R^6))$,  thus proving   Theorem \ref{2T2}.

Applying Theorem \ref{2T1} and a similar argument in Section 4,
we can show that $\gamma_\dl$ satisfies: $\gamma_\dl\in
L^\infty(0,T; L^p(\R^6))$ for all $T\in (0,\infty)$ and $1\le p<\infty$,
\begin{equation}\label{42}
0\le \gamma_\dl\le\gamma_\dl(f)\quad\textrm{a.e. on}\quad
\R^6\times(0,\infty),
\end{equation}
and
\begin{equation}\label{43}
\f{\partial \gamma_\dl}{\partial
t}+\Dv_x(\xi\gamma_\dl)+(E+\xi\times
B)\cdot\nabla_\xi\gamma_\dl=Q^{+}_\dl-Q^{-}_\dl,
\end{equation}
in $\mathcal{D}'$, where $Q^{+}_\dl$, $Q_\dl^{-}$ are respectively
the weak limits in $L^1(\R^3\times K\times(0,T))$ for all compact
sets $K\subset \R_\xi^3$ of $(1+\dl f^n)^{-1}Q^{+}(f^n,f^n)$,
$(1+\dl f^n)^{-1}Q^{-}(f^n,f^n)$. For the weak limit function $\gamma_\dl$,
we claim
\begin{Lemma}
$$\gamma_\dl\in C([0,\infty); L^p(\R^6))$$ for all $1\le p<\infty$.
\end{Lemma}
\begin{proof}
In fact, we claim that the weak limit $\gamma_\dl$ is a renormalized solution of
\eqref{43} and then $$\gamma_\dl\in C([0,\infty); L^p(\R^6))$$ for
all $1\le p<\infty$. For this purpose, we introduce
$$\gamma_\dl^\i(f^n)=\gamma_\dl(\beta_{\i}(f^n))$$ for $\i\in(0,1]$
and denote its weak limit by $\gamma^{\i}_{\dl}$. Then, the proof
 in Section 4 applies and shows that the weak limit $\gamma_{\dl}^{\i}\in
C([0,\infty);L^p(\R^6))$ is a renormalized solution of
\begin{equation}\label{44}
\begin{split}
\f{\partial}{\partial
t}\gamma_{\dl}^{\i}&+\Dv_x(\xi\gamma_{\dl}^{\i})+(E+\xi\times
B)\cdot\nabla_\xi\gamma_{\dl}^{\i}\\
&=\ov{\gamma'_\dl(\beta_\i(f))\beta'_\i(f)Q^{+}_{\dl,\i}}-\ov{\gamma'_\dl(\beta_\i(f))\beta'_\i(f)Q^{-}_{\dl,\i}},
\end{split}
\end{equation}
where the notation $\ov{g}$ means the weak limit of the sequence
$\{g_n\}_{n=1}^\infty$ in $L^1_{loc}$. Next, we claim
\begin{equation}\label{440}
0\le\gamma_\dl(f^n)-\gamma_{\dl}^{\i}(f^n)\le
f^n-\beta_{\i}(f^n)\rightarrow 0\quad\textrm{in}\quad L^1(\R^6)
\end{equation}
uniformly in $n\ge 1$, $t\in [0,T]$. Indeed, since the sequence
$\{f_n\}_{n=1}^\infty$ is equi-integrable, for any $\eta>0$, there
exists two positive numbers $D$ and $R$ such that
$$\sup_{n\in N}\int_{([0,T]\times B_R\times B_R)^c}f^ndtdxd\xi\le\eta,$$
and
$$\sup_{n\in N}\int_{\{f^n\ge D\}}f^ndtdxd\xi\le\eta.$$ Hence, in
particular,
$$\sup_{n\in N}\int_{\{f^n\ge D\}\cap [0,T]\times B_R\times
B_R}f^ndtdxd\xi\le\eta.$$ Therefore, we have
\begin{equation}\label{4400}
\begin{split}
&\sup_{n\in N}\int_{[0,T]\times
\R^3\times\R^3}(f^n-\beta_\i(f^n))dtdxd\xi \\
&\le\sup_{n\in N}\int_{([0,T]\times B_R\times
B_R)^c}f^ndtdxd\xi\\
&\quad+\sup_{n\in N}\int_{\{f^n\ge D\}\cap
[0,T]\times B_R\times B_R}f^ndtdxd\xi\\
&\quad+\sup_{n\in N}\int_{\{f^n\le D\}\cap [0,T]\times B_R\times
B_R}(f^n-\beta_\i(f^n))dtdxd\xi\\
&\le 2\eta+D^2R^6T\i,
\end{split}
\end{equation}
since $f^n-\beta_\i(f^n)\le D^2\i$ if $f^n\le D$. Thus, letting
first $\i$ go to $0$ and then $\eta$ go to $0$ in \eqref{4400}, we
deduce \eqref{440}.

Similarly, we have
$$0\le 1-\beta'_\i(f^n)\rightarrow 0\quad\textrm{in}\quad L^1(\R^6)$$
uniformly in $n\ge 1$, $t\in [0,T]$;
\begin{equation*}
\begin{split}
0&\le(\gamma'_\dl(\beta_{\i}(f^n))-\gamma'_{\dl}(f^n))\beta'_{\i}(f^n)Q^{-}(f^n,f^n)\\
&\le \f{\i f^n}{1+\dl f^n}Q^{-}(f^n, f^n)\rightarrow
0\quad\textrm{in}\quad L^1(0,T; L^1(\R^3_x\times K))
\end{split}
\end{equation*}
uniformly in $n\ge 1$, for all compact sets $K\subset \R^3_\xi$;
and
\begin{equation*}
\begin{split}
0&\le(\gamma'_\dl(\beta_{\i}(f^n)-\gamma'_{\dl}(f^n))\beta'_{\i}(f^n))Q^{+}(f^n,f^n)\\
&\le \f{\i f^n}{1+\dl f^n}Q^{+}(f^n, f^n)\rightarrow
0\quad\textrm{in}\quad L^1(0,T; L^1(\R^3_x\times K))
\end{split}
\end{equation*}
uniformly in $n\ge 1$, for all compact sets $K\subset \R^3_\xi$.
Here, we used
$$(\gamma'_\dl(\beta_{\i}(f^n))-\gamma'_{\dl}(f^n))\beta'_{\i}(f^n)\le\f{\i f^n}{1+\dl f^n}.$$
Thus, letting $\i$ go to $0$ in \eqref{44}, we deduce that
$\gamma_\dl$ is a renormalized solution to \eqref{43}.

Hence, from \eqref{43}, we deduce that
$$\f{\partial \gamma_\dl}{\partial t}\in L^1(0,T;
W^{-n,1}(\R^3))$$ for $n>0$ large enough. Also, we know that,
since $\gamma_\dl(t)$ is a strictly concave function,
$$0\le\gamma_\dl\le\gamma_\dl(f)\le f\in L^\infty(0,T;
L^1(\R^6)).$$ Hence, by the Aubin-Lions lemma in \cite{JL}, we
know that
$$\gamma_\dl\in C([0,T]; W^{-s, 1}(\R^6)).$$
But actually, we know
$$\gamma_\dl\in L^\infty([0,T], L^p(\R^6))$$
for all $1\le p<\infty$. Thus, by the interpolation, we know that
$$\gamma_\dl\in C([0,T]; L^p(\R^6))$$ for all $1\le p<\infty$.
\end{proof}

\subsection{Step One: $\gamma_\dl$ is a supersolution of \eqref{41}}
Without loss of generality, we may assume that we have
$$\gamma'_\dl(f^n)=\f{1}{1+\dl f^n}\rightarrow \zeta_\dl\quad
\textrm{weakly* in}\quad L^\infty(\R^6\times(0,\infty)),$$ and
$$f^n\gamma'_\dl(f^n)=\f{f^n}{1+\dl f^n}\rightarrow \theta_\dl\quad
\textrm{weakly* in}\quad L^\infty(\R^6\times(0,\infty)).$$
Furthermore, since $\gamma'_\dl(f)$,  $-t\gamma'_\dl(f)$ are convex
on $[0,\infty)$, we deduce the following inequalities:
\begin{subequations}\label{45}
\begin{align}
&\zeta_\dl\ge \f{1}{1+\dl f}=\gamma'_\dl(f),\quad \theta_\dl\le
\f{f}{1+\dl f}=f\gamma'_\dl(f),\\
&\gamma_\dl\le\f{1}{\dl}\ln(1+\dl
f)=\gamma_\dl(f)\quad\textrm{a.e. in}\quad \R^6\times(0,\infty).
\end{align}
\end{subequations}
We claim
\begin{Lemma}
\begin{equation}\label{46}
Q^{-}_\dl=\theta_\dl L(f)\quad\textrm{a.e. on}\quad
\R^6\times(0,\infty).
\end{equation}
\end{Lemma}
\begin{proof}
In fact, it is enough to verify that \eqref{46} holds in
$[0,T]\times B_R\times B_R$, where $B_R$ is the ball with radius
$R$ and centered at the origin in $\R^3$. Due to the second
statement of Theorem \ref{2T1}, we know that $L(f^n)$ converges
a.e. to $L(f)$ in $[0,T]\times B_R\times B_R$. By Egorov's Theorem
(\cite{Rubin}), for any $\i>0$, there exists a subset $E\subset
[0,T]\times B_R\otimes B_R$ with $|E|\le \i$ such that $L(f^n)$
converges uniformly to $L(f)$ on $E^c$. Thus, for all $\phi\in
L^\infty(\R^6\times (0,T))$,
\begin{equation*}
\begin{split}
&\left|\int_0^T\int_{B_R}\int_{B_R}\phi\left(\f{f^n}{1+\dl
f^n}L(f^n)-\theta_\dl L(f)\right) dxd\xi dt\right|\\
&\quad\le\|\phi\|_{L^\infty}\sup_{n}\int_{E}|L(f^n)|+|\theta_\dl||L(f)|dxd\xi
dt\\
&\qquad+\|\phi\|_{L^\infty}\left|\int_{E^c}\left(\f{f^n}{1+\dl f^n}-\theta_\dl\right) L(f)dxd\xi ds\right|\\
&\qquad+\|\phi\|_{L^\infty}|E^c|\sup_{E^c}\left|L(f^n)-L(f)\right|.
\end{split}
\end{equation*}
The first term  can be made arbitrarily small uniformly in $n$,
due to the equi-integrability of $\{L(f^n)\}_{n=1}^\infty$. The
second term also goes to $0$ since $L(f)\in L^1((0,T)\times
B_R\times B_R)$. And the third term goes to $0$ as $n$ goes to
$\infty$ since the uniform convergence of $L(f^n)$ to $L(f)$ in
$E^c$. Thus, \eqref{46} is verified.
\end{proof}

Similarly, we have

\begin{Lemma}
\begin{equation}\label{47}
Q_\dl^{+}=\zeta_\dl Q^{+}(f,f)\quad \textrm{a.e. on}\quad
\R^6\times(0,\infty).
\end{equation}
\end{Lemma}
\begin{proof}
Indeed, let $\mathcal{A}$ be an arbitrary compact subset of
$\R^6\times[0,\infty)$. By the Egorov's theorem and the fourth
statement of Theorem \ref{2T1},  for
each $\i>0$ there exists a measurable set $E$  with the  measure of $E$ not greater than $\i$ (i.e., $|E|\le\i$),
up to a subsequence $Q^{+}(f^n,f^n)$ converges uniformly to $Q^{+}(f,f)$ on $E^c$ and
$Q^{+}(f,f)$ is integrable on $E^c$. Then, for all $\phi\in
L^\infty(\R^6\times(0,\infty))$ supported in $\mathcal{A}$, we
have
\begin{equation*}
\begin{split}
&\left|\int_{\mathcal{A}}\phi\{\gamma'_\dl(f^n)Q^{+}(f^n,f^n)-\zeta_\dl
Q^{+}(f,f)\}dxd\xi dt\right|\\
&\le
\|\phi\|_{L^\infty}\int_E\left|\gamma'_\dl(f^n)Q^{+}(f^n,f^n)-\zeta_\dl
Q^{+}(f,f)\right|dxd\xi
dt\\
&\quad+\left|\int_{E^c\cap\mathcal{A}
}\phi\{\gamma'_{\dl}(f^n)-\zeta_\dl\}Q^+(f,f)dxd\xi
dt\right|\\
&\quad+\|\phi\|_{L^\infty}|E^c\cap \mathcal{A}|
\sup_{E^c}|Q^{+}(f^n,f^n)-Q^{+}(f,f)|,
\end{split}
\end{equation*}
where the third term goes to $0$ as $n$ goes to $\infty$, for each
$\i>0$ by the uniform convergence of $Q^+(f^n,f^n)$ to $Q^+(f,f)$
on $E^c$. And so does the second term since
$$\phi\chi_{E^c}Q^{+}(f,f)\in L^1(\R^6\times(0,\infty)).$$
 Finally, since
$\gamma'_\dl(f^n)Q^{+}(f^n,f^n)$ is weakly relatively compact in
$L^1(\R^3_x\times K\times (0,T))$ for all compact sets $K\subset
\R^3_\xi$, the first term can be made arbitrarily small uniformly
in $n$ if we let $\i$ go to $0$.

Notice also that $\zeta_\dl Q^{+}(f,f)\in L^1(\R^3_x\times K\times
(0,T))$ by following the similar argument as before, we can show
that $\zeta_\dl (Q^{+}(f,f)\wedge R)$ is the weak limit of
$\f{1}{1+\dl f^n}(Q^{+}(f^n,f^n)\wedge R)$, where $a\wedge
b=\min\{a,b\}$. Thus, \eqref{47} follows.
\end{proof}

Now, we use \eqref{45}-\eqref{47}  in \eqref{41} to obtain
\begin{equation}\label{48}
\f{\partial \gamma_\dl}{\partial
t}+\Dv_x(\xi\gamma_\dl)+(E+\xi\times B)\cdot\nabla \gamma_\dl\ge
\gamma'_\dl(f) Q(f,f)
\end{equation}
in $\mathcal{D}'$.
We conclude this first step by proving that $\gamma_\dl$ satisfies
the initial condition: $$\gamma_\dl|_{t=0}=\gamma_\dl(f_0).$$
Indeed, in view of the equation satisfied by $\gamma_\dl(f^n)$, we
know that $$\f{\partial \gamma_\dl(f^n)}{\partial t}\in L^1(0,T,
W^{-n, 1}(\R^6))$$ for $n>0$ large enough, which,  combined with the
fact $\gamma_\dl(f^n)\in L^\infty(0,T; L^1(\R^6))$ and the Aubin-Lions
lemma, implies that $$\gamma_\dl(f^n)\rightarrow \gamma_\dl
\qquad\textrm{ in}\qquad C([0,T]; W^{-s, 1}(\R^6))$$ for any
$s>1$. But, by the assumption,
$\gamma_\dl(f^n)|_{t=0}=\gamma_\dl(f^n_0)$ converges in
$L^1(\R^6)$ and thus in $W^{-s, 1}(\R^6)$ to $\gamma_\dl(f_0)$.
Thus, we conclude that $\gamma_\dl$ satisfies the initial
condition.

\bigskip
\subsection{Step Two: $\gamma_\dl=\gamma_\dl(f)$ and $f^n$ converges
in $L^1$ to $f$.} To this end, we consider
$$\gamma_\dl(f)-\gamma_\dl=\tau_\dl\in C([0,\infty); L^p(\R^6))$$
and  observe that $\tau_\dl$ satisfies,  in view of \eqref{41} and
\eqref{48},
\begin{equation}\label{49}
\f{\partial}{\partial t}\tau_\dl+\Dv_x(\xi \tau_\dl)+(E+\xi\times
B)\cdot\nabla_\xi \tau_\dl\le 0
\end{equation}
in $\mathcal{D}'$ with
\begin{equation}\label{410}
\tau_\dl\ge 0\quad \textrm{a.e. on}\quad
\R^6\times(0,\infty),\quad\tau_\dl|_{t=0}=0\quad\textrm{a.e.
on}\quad\R^6.
\end{equation}
Then,  for $\tau_\dl$ we have
\begin{Lemma}
$$\tau_\dl=0.$$
\end{Lemma}

\begin{proof}
Formally, we only need to integrate \eqref{49} over $\R^6$ to get
\begin{equation}\label{41000}
\f{d}{dt}\int_{\R^6}\tau_\dl dxd\xi\le 0\quad\textrm{in}\quad
\mathcal{D}'(0,\infty).
\end{equation}
Then \eqref{41000} with \eqref{410} yield: $\tau_\dl=0$
on $\R^6\times(0,\infty)$.

Our main objective now is to justify \eqref{41000}. In order to do
so, we introduce the function $\phi\in C_0^\infty(\R^3)$ with
\begin{equation*}
\phi(z)=
\begin{cases}
1,\qquad\textrm{if}\qquad |z|\le 1;\\
0,\qquad\textrm{if}\qquad |z|\ge 2.
\end{cases}
\end{equation*}
Notice that $\beta_{\i}(\tau_{\dl})=\f{\tau_{\dl}}{1+\i
\tau_{\dl}}$ also satisfies \eqref{49} and \eqref{410},  and
$\beta_{\i}(\tau_\dl)\in C([0,\infty); L^p(\R^6))$ for $1\le
p<\infty$ and $\i>0$, since $$|\beta_\i(x)-\beta_\i(y)|\le |x-y|$$
for all $x, y\ge 0$. Then we multiply \eqref{49} by
$\phi\left(\f{x}{n}\right)\phi\left(\f{\xi}{n}\right)$, and
integrate the resulting inequality over $\R^6\times(0,t)$ for all
$t\ge 0$ to obtain
\begin{equation}\label{411}
\begin{split}
&\int_{\R^6}\beta_{\i}(\tau_\dl)\phi\left(\f{x}{n}\right)\phi\left(\f{\xi}{n}\right)dxd\xi\\
&\quad\le\int_0^tds \int_{\R^6}dx d\xi \beta_{\i}(\tau_{\dl})
\cdot\Big(\f{\xi}{n}\cdot\nabla\phi\left(\f{x}{n}\right)\cdot\phi\left(\f{\xi}{n}\right)\\
&\qquad+\f{1}{n}(E+\xi\times
B)\cdot\nabla\phi\left(\f{\xi}{n}\right)\phi\left(\f{x}{n}\right)\Big).
\end{split}
\end{equation}

Recall that
\begin{equation}\label{L2}
\begin{split}
&\sup\left\{\int_{\R^6}\beta_{\i}(\tau_{\dl})|\xi|^2dxd\xi: \quad
t\in[0,T],\quad\i\ge 0,\quad \dl\ge 0\right\}\\
&\quad\le\sup_{t\in
[0,T]}\left\{\int_{\R^6}f|\xi|^2dxd\xi\right\}<\infty,
\end{split}
\end{equation}
since $\beta_\i(\tau_\dl)\le \tau_\dl\le\gamma_\dl(f)\le f$ for
all $T\in (0,\infty)$. Hence, for the terms on the right hand side
of \eqref{411}, we have
\begin{equation*}
\begin{split}
&\int_0^tds\int_{\R^6}dxd\xi
\beta_{\i}(\tau_{\dl})\left|\f{\xi}{n}\cdot\nabla\phi\left(\f{x}{n}\right)\right|\phi\left(\f{\xi}{n}\right)\\
&\quad\le \int_0^t\int_{\R^6}\beta_{\i}(\tau_{\dl})\chi_{\{n\le|x|\le
2n\}}2\|\phi\|_{L^\infty}\|\nabla\phi\|_{L^\infty}dxd\xi\\
&\quad\le 2\int_0^t\int_{\R^6}f\chi_{\{n\le|x|\le
2n\}}2\|\phi\|_{L^\infty}\|\nabla\phi\|_{L^\infty}dxd\xi\rightarrow
0
\end{split}
\end{equation*}
as $n\rightarrow \infty$ and
\begin{equation*}
\begin{split}
&\int_0^tds\int_{\R^6}dxd\xi
\beta_{\i}(\tau_{\dl})\f{1}{n}\phi\left(\f{x}{n}\right)\left|(E+\xi\times
B)\cdot \nabla\phi\left(\f{\xi}{n}\right)\right|\\
&\le \left(\int_0^tds \int_{\R^6}dx d\xi
\beta_{\i}(\tau_{\dl})\f{1}{n}|E+\xi\times B|\chi_{\{n\le|\xi|\le
2n\}}dxd\xi\right)\|\phi\|_{L^\infty}\|\nabla\phi\|_{L^\infty}.
\end{split}
\end{equation*}
Observing that, because of \eqref{L2},
\begin{equation*}
\begin{split}
\left\|\int_{\R^3}d\xi
\beta_{\i}(\tau_{\dl})\f{1}{n}\chi_{\{n\le|\xi|\le
2n\}}\right\|_{L^1(0,t;
L^1(\R^3_x))}&\le\f{1}{n^3}\left\|\int_{\R^3}d\xi
f|\xi|^2\chi_{\{n\le|\xi|\le 2n\}}\right\|_{L^1(0,t;
L^1(\R^3_x))}\\&=\f{1}{n^3}\i_n,\quad \textrm{with}\quad
\i_n\rightarrow 0,
\end{split}
\end{equation*}
while of course we have for some $C_{\i}>0$,
$$\left\|\int_{\R^3}d\xi
\beta_{\i}(\tau_{\dl})\f{1}{n}\chi_{\{n\le|\xi|\le
2n\}}\right\|_{L^1(0,t; L^\infty(\R^3_x))}\le C_{\i} n^2.$$
Therefore, we deduce from the H\"{o}lder inequality that we have for
all $\i>0$,
$$\int_{\R^3}d\xi \beta_{\i}(\tau_{\dl})\f{1}{n}\chi_{\{n\le|\xi|\le
2n\}}\to 0,$$
in $L^1(0,t; L^{p}(\R^3_x))$ for all $1\le p\le \f{5}{2}$ as $n\rightarrow \infty$, and hence
in particular in $L^1(0,t;L^{2}(\R^3))$. This implies
\begin{equation}\label{1117}
\int_0^tds\int_{\R^3}\int_{\R^3}\beta_\i(\tau_\dl)\f{1}{n}\chi_{\{n\le|\xi|\le
2n\}}d\xi |E|dx\rightarrow 0,\quad\textrm{in}\quad L^1((0,t)\times
\R^6)
\end{equation}
as $n\rightarrow \infty$, since $E\in L^\infty(0,t; L^2(\R^3_x))$.

On the other hand, using \eqref{1116} and the fact
$\beta_\i(\tau_\dl)\le f$, we obtain
\begin{equation}\label{1118}
\begin{split}
&\int_0^tds\int_{\R^6}dxd\xi
\beta_\i(\tau_\dl)\f{1}{n}|\xi||B|\chi_{\{n\le|\xi|\le 2n\}}\\
&\le\f{2}{n}\int_0^tds \left(\int_{\R^3}f|\xi|d\xi\right)|B|dx\\
&\quad\le \f{C}{n}\int_0^tds
\left(\int_{\R^3}f|\xi|^2d\xi\right)^{\f{4}{5}}|B|dx\\
&\quad\le
\f{C}{n}t\sup_{s\in(0,t)}\left\{\int_{\R^3}f|\xi|^2d\xi\right\}\|B\|_{L^\infty(0,t;
L^5(\R^3_x))}\\
&\quad \rightarrow 0,
\end{split}
\end{equation}
as $n\rightarrow \infty$.
Hence, combining \eqref{1117} and \eqref{1118} together,  we get
$$\int_0^tds \int_{\R^6}dxd\xi
\beta_{\i}(\tau_{\dl})\f{1}{n}|E+\xi\times B|\chi_{\{n\le|\xi|\le
2n\}}\rightarrow 0$$ as $n\rightarrow \infty$.

Finally, letting first $n$ go to $\infty$ and then $\i$ go to $0$
in \eqref{411}, we deduce, by Fatou's lemma,
$$\int_{\R^6}\tau_\dl(x,\xi,t)dxd\xi \le 0,\quad\textrm{for
all}\quad t\ge 0,$$
which, combined with \eqref{410}, implies that
$\tau_\dl=0$ on $\R^6\times(0,\infty)$ almost everywhere.
\end{proof}

In other words, $\gamma_\dl(f^n)$ weakly converges to
$\gamma_\dl(f)$. Since $\gamma_\dl$ is strictly concave on
$[0,\infty)$, we deduce from classical functional analysis
arguments that $f^n$ converges in measure to $f$ on
$\R^6\times(0,T)$ for all $T\in (0,\infty)$, see \cite{f2}.
This convergence implies that
\begin{equation}\label{413}
f^n\rightarrow f\qquad\textrm{in}\qquad L^p(0,T; L^1(\R^6)),
\end{equation}
for all $1\le p<\infty$ and $T\in (0,\infty)$. Indeed, by the
equi-integrability of the sequence $\{f^n\}_{n=1}^\infty$ and the
integrability of $f$ in $L^1([0,T]\times\R^6)$, we know that for
any $\i>0$, there exists $R>0$ and $\dl>0$ such that
\begin{equation}\label{4130}
\sup_{n\in {\mathbb N}}\int_{([0,T]\times B_R\times
B_R)^c}|f^n-f|dtdxd\xi\le\i,
\end{equation}
and
\begin{equation}\label{41300}
\sup_{n\in N}\int_{G}|f^n-f|dtdxd\xi\le\i,
\end{equation}
for all set $G\subset[0,T]\times\R^6$ with $|G|\le\dl$.

On the other hand, on the set $[0,T]\times B_R\times B_R$, since
up to a subsequence, $f^n$ converges to $f$ almost everywhere, by
Egorov's theorem, for the given $\dl>0$ as above, there exists a
subset $H\subset[0,T]\times B_R\times B_R$ with $|H|\le\dl$ such
that $f^n$ converges uniformly to $f$ on $([0,T]\times B_R\times
B_R)\cap H^c$. Therefore, using \eqref{41300},
\begin{equation}\label{413000}
\begin{split}
&\int_{[0,T]\times B_R\times B_R}|f^n-f|dtdxd\xi \\
&=\int_{([0,T]\times B_R\times B_R)\cap H^c}|f^n-f|dtdxd\xi+\int_{H}|f^n-f|dtdxd\xi\\
&\le\i+\int_{([0,T]\times B_R\times B_R)\cap H^c}|f^n-f|dtdxd\xi.
\end{split}
\end{equation}
Notice that the last term in \eqref{413000} tends to $0$ as
$n\rightarrow \infty$ since the uniform convergence of $f^n$ to
$f$ in $H^c$. Hence, combining \eqref{4130}, \eqref{41300} and
\eqref{413000}, we conclude that
$$f^n\rightarrow f\qquad\textrm{in}\qquad L^1(0,T; L^1(\R^6)),$$
which, with the uniform bound of $f^n$ in $L^\infty(0,T;
L^1(\R^6))$, implies \eqref{413}.

\bigskip
\subsection{Step Three: The convergence in $C([0,T]; L^1(\R^6))$.}
It  only remains to show that $f^n$ converges to $f$ in
$C([0,T]; L^1(\R^6))$ using \eqref{413}. Indeed, because of
\eqref{ee} and \eqref{313}, it is clearly enough to show that, for
each $\dl>0$, $T\in(0,\infty)$, $K$ compact set in $\R^6$, we have
\begin{equation}\label{414}
\beta_\dl(f^n)\rightarrow\beta_\dl(f)\quad\textrm{in}\quad
C([0,T]; L^1(K)).
\end{equation}

For this purpose, we take $\phi\in C_0^\infty(\R^6)$ such that
$\phi=1$ on $K$, $\phi\ge 0$,  and we use \eqref{32} to deduce that
for all $t\ge 0$,
\begin{equation}\label{415}
\begin{split}
\int_{\R^6}\beta_\dl(f^n)^2\phi dxd\xi=\int_0^t\int_{\R^6}dxd\xi
&\Big(\f{2\beta_\dl(f^n)}{1+\dl f^n}\times
Q(f^n,f^n)\phi\\&\quad+\beta_\dl(f^n)^2(\xi\cdot\nabla_x\phi+(E^n+\xi\times
B^n)\cdot\nabla_\xi\phi)\Big).
\end{split}
\end{equation}
Then, due to \eqref{413}, $\beta_\dl(f^n)$ converges to
$\beta_\dl(f)$ in $L^p(\R^6\times(0,T))$ for all $1\le p<\infty$
and $T\in (0,\infty)$,  and one can check easily that the right-hand
side of \eqref{415} converges uniformly in $t\in [0,T]$ to the
same expression with $f^n$ replaced by $f$. Since $\beta_\dl(f)$
is a renormalized solution, this expression is also given by
$\int_{\R^6}\beta_\dl(f)^2\phi dxd\xi$. In other
words, we have
\begin{equation}\label{416}
\begin{split}
\int_{\R^6}\beta_\dl(f^n)^2\phi dxd\xi\rightarrow
\int_{\R^6}\beta_\dl(f)^2\phi dxd\xi,
\end{split}
\end{equation}
uniformly in $t\in [0,T]$, for all $T\in (0,\infty)$.

In addition, since \eqref{32} implies $$\f{\partial
\beta_\dl(f^n)}{\partial t}\in L^1(0,T; W^{-n,1}(\R^6))$$ for
large enough $n>0$, and
$$\beta_\dl(f^n)\in L^1(0,T; L^1(\R^6)),$$
by the Aubin-Lions lemma, we know that $\beta_\dl(f^n)$ converges
to $\beta_\dl(f)$ in $C([0,T]; W^{-s,1}_{loc}(\R^6))$ for any
$s>1$. Therefore, if we consider
$L^2_{\phi}=L^2(\textrm{supp}\phi,\phi dx)$, since
$\{\beta_{\dl}(f^n)\}_{n}$ is bounded in
$L^\infty(0,T;L^2_{\phi})$, we deduce that $\beta_{\dl}(f^n)$
converges uniformly on $[0,T]$ to $\beta_\dl(f)$ in $L^2_{\phi}$
endowed with the weak topology, which, combined with \eqref{416}
and the fact that $\beta_\dl(f)\in C([0,\infty); L^2_{\phi})$
implies that $\beta_\dl(f^n)$ converges to $\beta_\dl(f)$ in
$L^2_{\phi}$ strongly and uniformly in $[0,T]$. Hence, \eqref{414}
follows.


\bigskip

\section{Large Time Behavior}
In this section, we are devoted to the study of the large time
behavior of the renormalized solution to VMB. Indeed, let
$f(t,x,\xi)$ be a renormalized solution to VMB with finite energy
and finite entropy in view of \eqref{ee}. Then, for every sequence
$\{t_n\}_{n=1}^\infty$ going to infinity, there exists a
subsequence $\{t_{n_k}\}_{k=1}^\infty$ and a local time-dependent
Maxwellian $m$ such that $f_{n_k}(t, x, \xi)= f(t+t_{n_k}, x,
\xi)$ converges weakly in $L^1((0,T)\times \R^6)$ to $m$ for every
$T>0$. More precisely, we have the following theorem:

\begin{Theorem}\label{3T1}
Let $f(t, x, \xi)$ be a renormalized solution to VMB and assume
that $b>0$ almost everywhere. Then, for every sequences $t_n$
going to infinity, there exists a subsequence $t_{n_k}$ and a
local time-dependent Maxwellian $m(t, x, \xi)$ such that
$f_{n_k}(t, x, \xi)=f(t+t_{n_k}, x, \xi)$ converges weakly in
$L^1((0,T)\times \R^6)$ to $m(t, x, \xi)$ for every $T>0$.
Moreover, the Maxwellian satisfies the Vlasov-Maxwell equations:
\begin{subequations}
\begin{align}
&\f{\partial m}{\partial t}+\xi\cdot\nabla_x m+(E+\xi\times
B)\cdot\nabla_\xi m=0,\\
&\f{\partial E}{\partial t}-\nabla\times B=-\int_{\R^3}m\xi d\xi,\qquad \Dv B=0,\\
&\f{\partial B}{\partial t}+\nabla\times E=0,\qquad \Dv
E=\int_{\R^3}md\xi,
\end{align}
\end{subequations}
in the sense of renormalizations.
\end{Theorem}
\begin{Remark}
When the spatial domain is a periodic box or a bounded domain with
the reverse reflexion boundary or the specular reflexion boundary,
we can expect, as in \cite{DD, LD}, that the local Maxwellian $m$
in Theorem \ref{3T1} is actually global; that is, $m$ is
independent of $t, x$.
\end{Remark}
\begin{Remark}
Our large time behavior result is only sequential; that is, the
Maxwellian could depend on our choice of the sequence
$\{t_n\}_{n=1}^\infty$.
\end{Remark}
\begin{proof}[Proof of Theorem \ref{3T1}]
Notice that since $f(t, x, \xi)$ is a renormalized solution to
VMB, it automatically holds:
\begin{equation}\label{52}
\begin{split}
&\sup_{t\in [0,\infty)}\left(\int_{\R^6}f(1+|\xi|^2+\nu(x)+|\log
f|)dxd\xi+\int_{\R^3}(|E|^2+|B|^2)dx\right)\\
&\quad+\int_0^\infty\int_{\R^3}dx\int_{\R^6}d\xi
d\xi_*\int_{S^2}d\omega b(f'f'_*-ff_*)\log
\f{f'f'_*}{ff_*}<\infty.
\end{split}
\end{equation}
Therefore, $f_n(t, x, \xi)=f(t+t_n, x, \xi)$ is weakly compact in
$L^1((0,T)\times \R^6)$ for every $T>0$ and each sequence of
positive numbers $\{t_n\}_{n=1}^\infty$ going to $\infty$.  Similarly, $E_n(t,x)=E(t+t_n, x)$, $B_n(t, x)=B(t+t_n, x)$ are
weakly compact in $L^\infty(0,T; L^2(\R^3))$.
 Then, the weak
compactness of $f_n(t, x, \xi)$ in $L^1((0,T)\times \R^3)$ implies
that there exists a subsequence $\{t_{n_k}\}_{k=1}^\infty$ and a
function $m\in L^1((0,T)\times \R^6)$ such that the function
$f_{n_k}$ converges weakly to $m$ in $L^1((0,T)\times \R^6)$ while
the weak compactness of $B_n(t, x)$ and $E_n(t, x)$ implies that we
can choose $t_{n_k}$ such that $B_{n_k}$ and $E_{n_k}$ converge
weakly* to $B$ and $E$ respectively in $L^\infty(0,T; L^2(\R^3))$.
Notice that, applying the velocity average lemma, we know
$$\int_{\R^3}f_{n_k}d\xi\rightarrow \int_{\R^3}m d\xi\qquad
\textrm{in}\qquad L^1(0,T; L^1(\R^3)),$$ and
$$\int_{\R^3}f_{n_k}\xi d\xi\rightarrow \int_{\R^3}m\xi d\xi\qquad
\textrm{in}\qquad L^1(0,T; L^1(\R^3)).$$ Hence, according to
\eqref{e212} and \eqref{e213}, the electric field $E$ and the
magnetic field $B$ satisfies
\begin{equation*}
\f{\partial E}{\partial t}-\nabla\times B=-\int_{\R^3}m\xi d\xi,
\end{equation*}
\begin{equation*}
\f{\partial B}{\partial t}+\nabla\times E=0,
\end{equation*}
with
$$\Dv B=0,\qquad \Dv E=\int_{\R^3}md\xi,$$
in the sense of distributions.

In order to prove that $m$ is a Maxwellian, we denote
\begin{equation*}
\begin{split}
&d_k:=\int_{0}^{T}\int_{\R^3}dx\int_{\R^6}d\xi
d\xi_*\int_{S^2}d\omega
b({f_{n_k}}'{f_{n_k}}'_*-{f_{n_k}}{f_{n_k}}_*)\log
\f{{f_{n_k}}'{f_{n_k}}'_*}{{f_{n_k}}{f_{n_k}}_*}\\
&\quad=\int_{t_{n_k}}^{T+t_{n_k}}\int_{\R^3}dx\int_{\R^6}d\xi
d\xi_*\int_{S^2}d\omega b(f'f'_*-ff_*)\log \f{f'f'_*}{ff_*}.
\end{split}
\end{equation*}
Then, the estimate \eqref{52} implies that $d_k$ converges to $0$
as $k$ goes to $\infty$.

On the other hand, in view of the first statement of Theorem
\ref{2T1} or arguing as \cite{DL5}, for all smooth nonnegative
functions $\psi, \phi$ with compact support, we have, up to a
subsequence,
\begin{equation}\label{53}
\begin{split}
&\int_{\R^6}d\xi d\xi_*\int_{S^2}d\omega
b{f_{n_k}}'{f_{n_k}}'_*\phi(\xi)\psi(\xi_*)\\
&\quad\rightarrow\int_{\R^6}d\xi d\xi_*\int_{S^2}d\omega bm(t, x,
\xi') m(t, x, \xi'_*) \phi(\xi)\psi(\xi_*),
\end{split}
\end{equation}
and
\begin{equation}\label{54}
\begin{split}
&\int_{\R^6}d\xi d\xi_*\int_{S^2}d\omega
b{f_{n_k}}{f_{n_k}}_*\phi(\xi)\psi(\xi_*)\\
&\quad\rightarrow\int_{\R^6}d\xi d\xi_*\int_{S^2}d\omega bm(t, x,
\xi) m(t, x, \xi_*) \phi(\xi)\psi(\xi_*),
\end{split}
\end{equation}
for almost all $(t,x)\in [0,T]\times \R^3$.

Furthermore, since $C(\R^3)$ is separable, we can also assume the
convergence in \eqref{53} and \eqref{54} holds for all nonnegative
function in $C(\R^3)$. Since $P(x,y)=(x-y)\ln (\f{x}{y})$ is a
nonnegative convex function for $x, y>0$, we have,
\begin{equation*}
\begin{split}
0&\le\int_{\R^6}d\xi d\xi_*\int_{S^2}d\omega \, b\,(m'm'_*-mm_*)\log
\f{m'm'_*}{mm_*}\psi(\xi_*)\phi(\xi)\\
&\le \liminf_{k\rightarrow \infty}d_k=0,
\end{split}
\end{equation*}
for almost all $(t,x)\in [0,T]\times \R^3$. Hence,
$$b(m'm'_*-mm_*)\log \f{m'm'_*}{mm_*}\psi(\xi')\phi(\xi)=0,$$
almost all $(t,x)\in [0,T]\times \R^3$. The nonnegativity of the
function $P(x,y)$ and the strict positivity of $b$ ensure that
$$m'm'_*=mm_*,$$
for almost all $(t,x, \xi, \xi_*, \omega)\in (0,T)\times \R^9\times
S^2$. According to Lemma 2.2 of \cite{FFM}, ] or  Section
3.2 of \cite{CC}, $m$ is a Maxwellian. Thus,
$$Q(m,m)=0.$$ Also, in view of Theorem \ref{2T1}, $m$ is still a
renormalized solution to VMB,   hence
\begin{equation*}
\f{\partial m}{\partial t}+\xi\cdot\nabla_x m+(E+\xi\times
B)\cdot\nabla_\xi m=0,
\end{equation*}
in the sense of renormalizations.
The proof is complete.
\end{proof}

\bigskip

\section{Remark on The Relativistic Vlasov-Maxwell-Boltmann Equations}

An extension of our analysis is possible to the relativistic
Vlasov-Maxwell-Boltzmann equations of the form (cf. \cite{CG, RTG1}):
\begin{subequations}\label{rvmb1111}
\begin{align}
&\f{\partial f}{\partial t}+\hat{\xi}\cdot\nabla_x
f+(E+\hat{\xi}\times B)\cdot\nabla_\xi f=Q(f,f),\quad x\in
\R^3,\quad \xi\in \R^3,\quad
t\ge 0,\\
&\f{\partial E}{\partial t}-\nabla\times B=-j, \quad \Dv B=0,\quad
\textrm{on}\quad \R^3_x\times(0,\infty),\\
&\f{\partial B}{\partial t}+\nabla\times E=0, \quad\Dv E=\r,\quad
\textrm{on}\quad \R^3_x\times(0,\infty),\\
&\r=\int_{\R^3}fd\xi, \quad j=\int_{\R^3}f\hat{\xi}
d\xi,\quad \textrm{on}\quad
\R^3_x\times(0,\infty),
\end{align}
\end{subequations}
with $$\hat{\xi}:=\f{\xi}{\sqrt{1+|\xi|^2}},$$ and
$$Q(f,f)=\int_{\R^3}d\xi_*\int_{S^2}d\omega
\f{b(\xi-\xi_*,\omega)}{\sqrt{1+|\xi|^2}\sqrt{1+|\xi_*|^2}}(f'f'_*-ff_*).$$
The corresponding conservation laws are given by
\begin{equation*}
\f{\partial \r}{\partial t}+\Dv_xj=0,
\end{equation*}
\begin{equation*}
\begin{split}
&\f{\partial }{\partial t}\left(\int_{\R^3}f\xi d\xi+E\times
B\right)\\
&\quad+\Dv_x\left(\int_{\R^3}\xi\otimes\hat{\xi}
fd\xi+\left(\f{|E|^2+|B|^2}{2}Id-E\otimes E-B\otimes B\right)
\right)=0,
\end{split}
\end{equation*}
and
\begin{equation*}
\begin{split}
&\f{\partial}{\partial
t}\left(\int_{\R^3}f\sqrt{1+|\xi|^2}d\xi+|E(t,x)|^2+|B(t,x)|^2\right)\\
&\quad\qquad+\Dv_x\left(\int_{\R^3}\hat{\xi}|\xi|^2
fd\xi+2E(t,x)\times B(t,x)\right)=0.
\end{split}
\end{equation*}
Then,  we can deduce that for any $t\in [0,T]$
\begin{equation}\label{rvmb1}
\int_{\R^6}f(t,x,\xi)(\sqrt{1+|\xi|^2}+\sqrt{1+|x|^2})dxd\xi\le
C(T),
\end{equation}
which implies $f|\xi|\in L^\infty(0,T; L^1(\R^6)$. Hence, following
the lines in Section 2, the existence of renormalized solution to
the relativistic Vlasov equation can be verified. More importantly, we
can further release the requirement on the integrability of the
electric field in $L^5$, since we no longer need the estimate on
$f|\xi|^2$ in $L^\infty(0,T; L^1(\R^6))$.

Note that for the relativistic VMB, the magnetic field has the same
integrability in the variable $x$ as the magnetic field due to
equivalence between $\xi$ and $\sqrt{1+|\xi|^2}$ when $\xi$ is
sufficiently large. More precisely, we have
\begin{Proposition}\label{rvmb2}
Assume that $f\in L^\infty((0,T)\times\R^6)$. Then for any
solution satisfying the above conservation laws, one has
$$\|\r(t,x)\|_{L^\infty(0,T; L^{\f{4}{3}}(\R^3))}\le C,\quad \|j(t,x)\|_{L^\infty(0,T; L^{\f{4}{3}}(\R^3))}\le
C,$$ where the constant $C$ depends on the energy of the initial
data and on $\|f\|_{L^\infty((0,T)\times\R^6)}$.
\end{Proposition}
\begin{proof}
Indeed, we have
\begin{equation*}
\begin{split}
\r(t,x)&=\int_{|\xi|\le R}f(t,x,\xi)d\xi+\int_{|\xi|\le
R}f(t,x,\xi)d\xi\\
&\le
C\f{1}{R^3}\|f\|_{L^\infty}+R^{-1}\int_{\R^3}f\sqrt{1+|\xi|^2}d\xi\\
&\le C\left(\int_{\R^3}f\sqrt{1+|\xi|^2}d\xi\right)^{\f{3}{4}}
\end{split}
\end{equation*}
where for the last inequality, we optimize $R$ by taking
$$R=\left(\int_{\R^3}f\sqrt{1+|\xi|^2}d\xi\right)^{\f{1}{4}}.$$

The same computation also works for $j$.
\end{proof}

For any sequence of $f^n$ as in Section 3, by the H Theorem and
\eqref{rvmb1}, $f^n$ is weakly compact in $L^1((0,T)\times(\R^6)$.
And then, we can follow the lines in Section 4 and Section 5 to
show the corresponding weak stability for the relativistic VMB. One
difference is that, due to Proposition \ref{rvmb2}, we need to
assume the electric field $E(t,x)$, and the magnetic field
$B(t,x)$ are uniformly bound in $L^\infty(0,T; L^\alpha(\R^3))$
for some $\alpha>4$. When the weak stability and the existence of
renormalized solutions to \eqref{rvmb1111} are concerned, a
different assumption on the collision kernel need to assume, that
is,
$$(1+|z|^2)^{-1}\left(\int_{z+B_R}\f{A(\xi)}{\sqrt{1+|\xi|^2}}d\xi\right)\rightarrow
0,\quad \textrm{as}\quad |z|\rightarrow\infty,\quad \textrm{for
all}\quad R\in(0,\infty).$$

\bigskip\bigskip

\section*{Acknowledgments}

D. Wang's research was
supported in part by the National Science Foundation under Grant
DMS-0906160 and by the Office of Naval Research
under Grant N00014-07-1-0668.


\bigskip\bigskip

\begin{thebibliography}{999}

\bibitem{AV} Alexandre, R.; Villani, C.,
\emph{On the Boltzmann equation for long-range interactions}.
Comm. Pure Appl. Math. 55 (2002), 30--70.

\bibitem{AL} Ambrosio, L., \emph{Transport equation and Cauchy problem
for $BV$ vector fields.} Invent. Math. 158 (2004), 227--260.

\bibitem{FB} Bouchut, F.,
\emph{Renormalized solutions to the Vlasov equation with
coefficients of bounded variation}. Arch. Ration. Mech. Anal. 157
(2001), 75--90.

\bibitem{FFM} Bouchut, F.; Golse, F.; Pulvirenti, M.,
\emph{Kinetic equations and asymptotic theory}. Series in Applied
Mathematics (Paris), 4. Gauthier-Villars, \'{E}ditions Scientifiques et M\'{e}dicales Elsevier, Paris, 2000.

\bibitem{CC} Cercignani, C.; Illner, R.; Pulvirenti, M.,
\emph{The mathematical theory of dilute gases}. Applied
Mathematical Sciences, 106. Springer-Verlag, New York, 1994.

\bibitem{CG} Cercignani, C.; Kremer, G. M.,
\emph{The relativistic Boltzmann equation: theory and
applications}. Progress in Mathematical Physics, 22. Birkhauser
Verlag, Basel, 2002.

\bibitem{DD} Desvillettes, L.; Dolbeault, J.,
\emph{On long time asymptotics of the Vlasov-Poisson-Boltzmann
equation}. Comm. Partial Differential Equations 16 (1991),
451--489.

\bibitem{DL} DiPerna, R. J.; Lions, P.-L.,
\emph{Ordinary differential equations, transport theory and
Sobolev spaces}. Invent. Math. 98 (1989), 511--547.

\bibitem{DL1} DiPerna, R. J.; Lions, P.-L.,
\emph{Global weak solutions of Vlasov-Maxwell systems}. Comm. Pure
Appl. Math. 42 (1989), 729--757.

\bibitem{DL4} DiPerna, R. J.; Lions, P.-L.,
\emph{On the Cauchy problem for Boltzmann equations: global
existence and weak stability}. Ann. of Math. (2) 130 (1989),
321--366.

\bibitem{DL5} DiPerna, R. J.; Lions, P.-L.,
\emph{Global solutions of Boltzmann's equation and the entropy
inequality}. Arch. Rational Mech. Anal. 114 (1991), 47--55.

\bibitem{LD} Desvillettes, L.,
\emph{Convergence to equilibrium in large time for Boltzmann and
B.G.K. equations}. Arch. Rational Mech. Anal. 110 (1990), 73--91.

\bibitem{DLM} DiPerna, R. J.; Lions, P.-L.; Meyer, Y.,
\emph{$L\sp p$ regularity of velocity averages}. Ann. Inst. H.
Poincar�� Anal. Non Lin��aire 8 (1991), 271--287.

\bibitem{f2} Feireisl, E., \emph{ Dynamics of viscous compressible
fluids.} Oxford Lecture Series in Mathematics and its
Applications, 26. Oxford University Press, Oxford, 2004.

\bibitem{LP}  Lifschitz, E. M.;  Pitaevskii,  L. P., \emph{Physical
kinetics}, Vol. 10 of Course of Theoretical Physics, Pergamon
Press, Oxford, 1981.

\bibitem{RTG} Glassey, R. T.,
\emph{The Cauchy problem in kinetic theory}. Society for
Industrial and Applied Mathematics (SIAM), Philadelphia, PA, 1996.

\bibitem{RTG1} Glassey, R. T., \emph{Global solutions to the Cauchy problem
for the relativistic Boltzmann equation with near-vacuum data.}
Comm. Math. Phys. 264 (2006), 705--724.

\bibitem{YG} Guo, Y.,
\emph{The Vlasov-Maxwell-Boltzmann system near Maxwellians}.
Invent. Math. 153 (2003), 593--630.

\bibitem{HK} Hamdache, K.,
\emph{Initial-boundary value problems for the Boltzmann equation:
global existence of weak solutions}. Arch. Rational Mech. Anal.
119 (1992), 309--353.

\bibitem{LL} Le Bris, C.; Lions, P.-L.,
\emph{Renormalized solutions of some transport equations with
partially $W\sp {1,1}$ velocities and applications}. Ann. Mat.
Pura Appl. (4) 183 (2004), 97--130.

\bibitem{JL} Lions, J.-L., 
{\em Quelques mŽthodes de rŽsolution des problmes aux limites non linŽaires.} (French) Dunod; Gauthier-Villars, Paris 1969.

\bibitem{DL2} Lions, P.-L.,
\emph{Compactness in Boltzmann's equation via Fourier integral
operators and applications}. I, II. J. Math. Kyoto Univ. 34
(1994), 391--427, 429--461.

\bibitem{DL3} Lions, P.-L.,
\emph{Compactness in Boltzmann's equation via Fourier integral
operators and applications}. III. J. Math. Kyoto Univ. 34 (1994),
539--584.

\bibitem{GR} Rein, G.,
\emph{Global weak solutions to the relativistic Vlasov-Maxwell
system revisited}. Commun. Math. Sci. 2 (2004), 145--158.

\bibitem{Rubin} Rudin, W.,
{\em Functional analysis.} Second edition. International Series in Pure and Applied Mathematics. McGraw-Hill, Inc., New York, 1991. 

\bibitem{SJ} Schaeffer, J.
\emph{The classical limit of the relativistic Vlasov-Maxwell
system}. Comm. Math. Phys. 104 (1986), 403--421.

\bibitem{SR} Strain, R.  M.,
\emph{The Vlasov-Maxwell-Boltzmann system in the whole space}.
 Comm. Math. Phys. 268 (2006), 543--567.

\bibitem{VC} Villani, C.,
\emph{On the Cauchy problem for Landau equation: sequential
stability, global existence}. Adv. Differential Equations 1
(1996), 793--816.

\end{thebibliography}
\end{document}